\documentclass{imsart}

\RequirePackage[OT1]{fontenc}
\RequirePackage{amsthm,amsmath,natbib}
\usepackage{bm}
\usepackage{amsfonts}
\usepackage{amssymb}
\usepackage{upgreek}
\usepackage{mathrsfs}
\usepackage{ctable}
\usepackage{times}
\usepackage{multirow}
\usepackage{url}
\startlocaldefs
\numberwithin{equation}{section}
\theoremstyle{plain}
\newtheorem{theorem}{Theorem}[section]
\theoremstyle{remark}
\newtheorem{remark}{Remark}[section]
\newtheorem{lemma}{Lemma}
\newtheorem{corollary}{Corollary}
\newcommand{\Var}{\mathrm{Var}}
\newcommand{\E}{\mathrm{E}}
\endlocaldefs

\begin{document}

\begin{frontmatter}
\title{Efficient Bayesian estimation and uncertainty quantification in ordinary differential equation models}
\runtitle{Efficient Bayesian estimation in ODE models}

\begin{aug}
\author{\fnms{Prithwish} \snm{Bhaumik}\thanksref{a1,e1}\ead[label=e1,mark]{prithwish.bhaumik@utexas.edu}}
\author{\fnms{Subhashis} \snm{Ghosal}\thanksref{a2,e2}\ead[label=e2,mark]{sghosal@ncsu.edu}}
%\and
%\author{\fnms{Third} \snm{Author}\thanksref{b}%
%\ead[label=e3]{third@somewhere.com}%
%\ead[label=u1,url]{www.foo.com}}
\address[a1]{Department of Statistics and Data Sciences, The University of Texas at Austin, Austin, TX 78712, USA.
\printead{e1}}

\address[a2]{Department of Statistics, North Carolina State University, Raleigh, NC 27695, USA.
\printead{e2}}

%\address[b]{Address of the Second and Third author,
%usually few lines long,
%usually few lines long.
%\printead{e3},
%\printead{u1}}

\runauthor{P. Bhaumik and S. Ghosal}

\affiliation{North Carolina State University}

\end{aug}

\begin{abstract}
Often the regression function is specified by a system of ordinary differential equations (ODEs) involving some unknown parameters. Typically analytical solution of the ODEs is not available, and hence likelihood evaluation at many parameter values by numerical solution of equations may be computationally prohibitive.  \citet{bhaumik2014bayesiant} considered a Bayesian two-step approach by embedding the model in a larger nonparametric regression model, where a prior is put through a random series based on B-spline basis functions. A posterior on the parameter is induced from the regression function by minimizing an integrated weighted squared distance between the derivative of the regression function and the derivative suggested by the ODEs. Although this approach is computationally fast, the Bayes estimator is not asymptotically efficient. In this paper we suggest a modification of the two-step method by directly considering the distance between the function in the nonparametric model and that obtained from a four stage Runge-Kutta (RK$4$) method. We also study the asymptotic behavior of the  posterior distribution of $\bm\theta$ based on an approximate likelihood obtained from an RK$4$ numerical solution of the ODEs. We establish a Bernstein-von Mises theorem for both methods which assures that Bayesian uncertainty quantification matches with the frequentist one and the Bayes estimator is asymptotically efficient.
\end{abstract}



\end{frontmatter}

\section{Introduction}
Differential equations are encountered in various branches of science such as in genetics \citep{chen1999modeling}, viral dynamics of infectious diseases
[\citet{anderson1992infectious}, \citet{nowak2000virus}], pharmacokinetics and pharmacodynamics (PKPD) \citep{gabrielsson2000pharmacokinetic}. In many cases these equations do not lead to any explicit solution. A popular example is the Lotka-Volterra equations, also known as predator-prey equations. The rates of change of the prey and predator populations are given by the equations
\begin{eqnarray}
\frac{df_{1\bm\theta}(t)}{dt}&=&\theta_1f_{1\bm\theta}(t)-\theta_2f_{1\bm\theta}(t)f_{2\bm\theta}(t),\nonumber\\
\frac{df_{2\bm\theta}(t)}{dt}&=&-\theta_3f_{2\bm\theta}(t)+\theta_4f_{1\bm\theta}(t)f_{2\bm\theta}(t),\,\,t\in[0,1],\nonumber
\end{eqnarray}
where $\bm\theta=(\theta_1,\theta_2,\theta_3,\theta_4)^T$ and $f_{1\bm\theta}(t)$ and $f_{2\bm\theta}(t)$ denote the prey and predator populations at time $t$ respectively. These models can be put in a regression model $\bm Y=\bm f_{\bm\theta}(t)+\bm\varepsilon,\,\bm\theta\in\Theta\subseteq\mathbb{R}^p$, where $\bm f_{\bm\theta}(\cdot)$ satisfies the ODE
\begin{eqnarray}
\frac{d\bm f_{\bm\theta}(t)}{dt}=\bm F(t,\bm f_{\bm\theta}(t),\bm\theta),\,t\in [0,1];\label{intro}
\end{eqnarray}
here $\bm F$ is a known appropriately smooth vector valued function and $\bm\theta$ is a parameter vector controlling the regression function.\\
\indent The nonlinear least squares (NLS) is the usual way to estimate the unknown parameters provided that the analytical solution of the ODE is available, which is not the case in most real applications. The 4-stage Runge-Kutta algorithm (RK4)  [\citet[page 134]{hairer1993solving} and \citet[page 53]{mattheij2002ordinary}] can solve \eqref{intro} numerically. The parameters can be estimated by applying NLS in the next step. \citet{xue2010sieve} studied the asymptotic properties of the estimator. They used differential evolution method \citep{storn1997differential}, scatter search method and sequential quadratic programming \citep{rodriguez2006novel} method for the NLS part and established the strong consistency, $\sqrt{n}$-consistency and asymptotic normality of the estimator. The estimator turns out to be asymptotically efficient, but this approach is computationally intensive.\\
\indent In the generalized profiling procedure \citep{ramsay2007parameter}, a linear combination of basis functions is used to obtain an approximate solution. A penalized optimization is used to estimate the coefficients of the basis functions. The estimated $\bm\theta$ is defined as the maximizer of a data dependent fitting criterion involving the estimated coefficients. The statistical properties of the estimator obtained from this approach were explored in the works of \citet{qi2010asymptotic}. This method is also asymptotically efficient, but has a high computational cost.\\
\indent \citet{varah1982spline} used a two-step procedure where the state variables are approximated by cubic spline in the first step. In the second step, the parameters are estimated by minimizing the sum of squares of difference between the non-parametrically estimated derivative and the derivatives suggested by the ODEs at the design points. Thus the ODE model is embedded in the nonparametric regression model. This method is very fast and independent of the initial or boundary conditions. \citet{brunel2008parameter} did a modification by replacing the sum of squares by a weighted integral and obtained the asymptotic normality of the estimator. \citet{gugushvili2012n} followed the approach of \citet{brunel2008parameter}, but used kernel smoothing instead of spline and established $\sqrt{n}$-consistency of the estimator. \citet{wu2012numerical} used penalized smoothing spline in the first step and numerical derivatives of the nonparametrically estimated functions. \citet{brunel2014parametric} used nonparametric smoothing and a set of orthogonality conditions to estimate the parameters. But the major drawback of the two-step estimation methods is that these are not asymptotically efficient.\\
\indent ODE models in Bayesian framework was considered in the works of \citet{gelman1996physiological}, \citet{rogers2007bayesian} and \citet{girolami2008bayesian}. They obtained an approximate likelihood by solving the ODEs numerically. Using the prior assigned on $\bm\theta$, MCMC technique was used to generate samples from the posterior. This method also has high computational complexity. \citet{campbell2012smooth} proposed the smooth functional tempering approach which utilizes the generalized profiling approach \citep{ramsay2007parameter} and the parallel tempering algorithm. \citet{jaeger2009functional} also used the generalized profiling in Bayesian framework. \citet{bhaumik2014bayesiant} considered the Bayesian analog of the two-step method suggested by \citet{brunel2008parameter}, putting prior on the coefficients of the B-spline basis functions and induced a posterior on $\Theta$. They established a Bernstein-von Mises theorem for the posterior distribution of $\bm\theta$ with $n^{-1/2}$ contraction rate.\\
\indent In this paper we propose two separate approaches. We use Gaussian distribution as the working model for error, although the true distribution may be different. The first approach involves assigning a direct prior on $\bm\theta$ and then constructing the posterior of $\bm\theta$ using an approximate likelihood function constructed using the approximate solution $\bm f_{\bm\theta,r}(\cdot)$ obtained from RK$4$. Here $r$ is the number of grid points used. When $r$ is sufficiently large, the approximate likelihood is expected to behave like the actual likelihood. We call this method Runge-Kutta sieve Bayesian (RKSB) method. In the second approach we define $\bm\theta$ as $\arg\min_{\bm\eta\in\bm\Theta}\int_0^1\|\bm\beta^T\bm N(\cdot)-\bm f_{\bm\eta,r}(\cdot)\|^2w(t)dt$ for an appropriate weight function $w(\cdot)$ on $[0,1]$, where the posterior distribution of $\bm\beta$ is obtained in the nonparametric spline model and $\bm N(\cdot)$ is the B-spline basis vector. We call this approach Runge-Kutta two-step Bayesian (RKTB) method. Thus, this approach is similar in spirit to \citet{bhaumik2014bayesiant}. Similar to \citet{bhaumik2014bayesiant}, prior is assigned on the coefficients of the B-spline basis and the posterior of $\bm\theta$ is induced from the posterior of the coefficients. But the main difference lies in the way of extending the definition of parameter. Instead of using deviation from the ODE, we consider the distance between function in the nonparametric model and RK$4$ approximation of the model. \citet{sujit2010statistical} considered Euler's approximation to construct the approximate likelihood and then drew posterior samples. In the same paper they also provided a non-Bayesian method by estimating $\bm\theta$ by minimizing the sum of squares of the difference between the spline fitting and the Euler approximation at the grid points. However they did not explore the theoretical aspects of those methods. We shall show both RKSB and RKTB lead to Bernstein-von Mises Theorem with dispersion matrix inverse of Fisher information and hence both the proposed Bayesian methods are asymptotically efficient. This was not the case for the two step-Bayesian approach \citep{bhaumik2014bayesiant}. Bernstein-von Mises Theorem implies that credible intervals have asymptotically correct frequentist coverage. The computational cost of the two-step Bayesian method \citep{bhaumik2014bayesiant} is the least, RKTB is more computationally involved and RKSB is the most computationally expensive.\\
\indent The paper is organized as follows. Section 2 contains the description of the notations and some preliminaries of Runge-Kutta method. The model assumptions and prior specifications are given in Section $3$. The main results are given in Section $4$. In Section 5 we carry out a simulation study. Proofs of the main results are given in Section 6. Section 7 contains the proofs of the technical lemmas. The appendix is provided in Section 8.
\section{Notations and preliminaries}
We describe a set of notations to be used in this paper. Boldfaced letters are used to denote vectors and matrices. The identity matrix of order $p$ is denoted by $\bm I_p$. We use the symbols $\mathrm{maxeig}(\bm A)$ and $\mathrm{mineig}(\bm A)$ to denote the maximum and minimum eigenvalues of the matrix $\bm A$ respectively.
The $L_2$ norm of a vector $\bm x\in\mathbb{R}^p$ is given by $\|\bm x\|=\left(\sum_{i=1}^px_i^2\right)^{1/2}$. The notation $f^{(r)}(\cdot)$ stands for the $r^{th}$ order derivative of a function $f(\cdot)$, that is, $f^{(r)}(t)=\frac{d^r}{dt^r}f(t)$. For the function $\bm\theta\mapsto f_{\bm\theta}(x)$, the notation $\dot{f}_{\bm\theta}(x)$ implies $\frac{\partial}{\partial\bm\theta}f_{\bm\theta}(x)$. Similarly, we denote $\ddot{f}_{\bm\theta}(x)=\frac{\partial^2}{\partial\bm\theta^2}f_{\bm\theta}(x)$. A vector valued function is represented by the boldfaced symbol $\bm f(\cdot)$. We use the notation $f(\bm x)$ to denote the vector $(f(x_1),\ldots,f(x_p))^T$ for a real-valued function $f:[0,1]\rightarrow\mathbb{R}$ and a vector $\bm x\in\mathbb{R}^p$. Let us define $\|\bm f\|_g=(\int_0^1\|\bm f(t)\|^2g(t)dt)^{1/2}$ for $\bm f:[0,1]\mapsto\mathbb{R}^p$ and $g:[0,1]\mapsto[0,\infty)$. The weighted inner product with the corresponding weight function $g(\cdot)$ is denoted by $\langle\cdot,\cdot\rangle_g$. For numerical sequences $a_n$ and $b_n$, both $a_n=o(b_n)$ and $a_n\ll b_n$ mean $a_n/b_n\rightarrow0$ as $n\rightarrow\infty$. Similarly we define $a_n\gg b_n$. The notation $a_n=O(b_n)$ is used to indicate that $a_n/b_n$ is bounded. The notation $a_n\asymp b_n$ stands for both $a_n=O(b_n)$ and $b_n=O(a_n)$, while $a_n\lesssim b_n$ means $a_n=O(b_n)$. The symbol $o_P(1)$ stands for a sequence of random vectors converging in $P$-probability to zero, whereas $O_P(1)$ stands for a stochastically bounded sequence of random vectors. Given a sample $\{X_i:\,i=1,\ldots,n\}$ and a measurable function $\psi(\cdot)$, we define $\mathbb{P}_n\psi=n^{-1}\sum_{i=1}^n\psi(X_i)$. The symbols $\E(\cdot)$ and $\Var(\cdot)$ stand for the mean and variance respectively of a random variable, or the mean vector and the variance-covariance matrix of a random vector. We use the notation $\mathbb{G}_n\psi$ to denote $\sqrt{n}\left(\mathbb{P}_n\psi-\E\psi\right)$. For a measure $P$, the notation $P^{(n)}$ implies the joint measure of a random sample $X_1,\ldots,X_n$ coming from the distribution $P$. Similarly we define $p$ and $p^{(n)}$ for the corresponding densities. The total variation distance between the probability measures $P$ and $Q$ defined on $\mathbb{R}^p$ is given by $\|P-Q\|_{TV}=\sup_{B\in\mathscr{R}^p}|P(B)-Q(B)|$, $\mathscr{R}^p$ being the Borel $\sigma$-field on $\mathbb R^p$. Given an open set $E$, the symbol $C^m(E)$ stands for the class of functions defined on $E$ having first $m$ continuous partial derivatives with respect to its arguments. For a set $A$, the notation $\mathrm{l}\!\!\!1\{A\}$ stands for the indicator function for belonging to $A$. The symbol $:=$ means equality by definition. For two real numbers $a$ and $b$, we use the notation $a\wedge b$ to denote the minimum of $a$ and $b$. Similarly we denote $a\vee b$ to be the maximum of $a$ and $b$.\\
\indent Given $r$ equispaced grid points $a_1=0,a_2,\ldots,a_r$ with common difference $h$ and an initial condition $\bm f_{\bm\theta}(0)=\bm y_0$, Euler's method \citep[page 9]{henrici1962discrete} computes the approximate solution as $\bm f_{\bm\theta,r}(a_{k+1})=\bm f_{\bm\theta,r}(a_{k})+h\bm F(a_k,\bm f_{\bm\theta,r}(a_{k}),\bm\theta )$ for $k=1,2,\ldots,r-1$. The RK$4$ method \citep[page 68]{henrici1962discrete} is an improvement over Euler's method. Let us denote
\begin{eqnarray}
\bm k_1&=&\bm F(a_k,\bm f_{\bm\theta,r}(a_{k}),\bm\theta ),\nonumber\\
\bm k_2&=&\bm F(a_k+h/2,\bm f_{\bm\theta,r}(a_{k})+h/2\bm k_1,\bm\theta ),\nonumber\\
\bm k_3&=&\bm F(a_k+h/2,\bm f_{\bm\theta,r}(a_{k})+h/2\bm k_2,\bm\theta ),\nonumber\\
\bm k_4&=&\bm F(a_k+h,\bm f_{\bm\theta,r}(a_{k})+h\bm k_3,\bm\theta ).\nonumber
\end{eqnarray}
Then we obtain $\bm f_{\bm\theta,r}(a_{k+1})$ from $\bm f_{\bm\theta,r}(a_{k})$ as $\bm f_{\bm\theta,r}(a_{k+1})=\bm f_{\bm\theta,r}(a_{k})+h/6(\bm k_1+2\bm k_2+2\bm k_3+\bm k_4)$. By the proof of Theorem 3.3 of \citet[page 124]{henrici1962discrete}, we have
\begin{eqnarray}
\sup_{t\in[0,1]}\|\bm f_{\bm\theta}(t)-\bm f_{\bm\theta,r}(t)\|=O(r^{-4}),\,
\sup_{t\in[0,1]}\left\|\dot{\bm f}_{\bm\theta}(t)-\dot{\bm f}_{\bm\theta,r}(t)\right\|=O(r^{-4}).\label{rk4}
\end{eqnarray}

\section{Model assumptions and prior specifications}
Now we formally describe the model. For the sake of simplicity we assume the response to be one dimensional. The extension to the multidimensional case is straight forward. The proposed model is given by
\begin{eqnarray}
Y_{i}=f_{\bm\theta}(X_i)+\varepsilon_{i},\,i=1,\ldots,n,\label{prop}
\end{eqnarray}
where $\bm\theta\subseteq\Theta$, which is a compact subset of $\mathbb{R}^p$. The function $ f_{\bm\theta}(\cdot)$ satisfies the ODE given by
\begin{eqnarray}
\frac{df_{\bm\theta}(t)}{dt}=F(t,f_{\bm\theta}(t),\bm\theta),\,t\in[0,1].\label{diff}
\end{eqnarray}
Let for a fixed $\bm\theta$, $F\in C^{m-1}((0,1),\mathbb{R})$ for some integer $m\geq 1$. Then, by successive differentiation we have $f_{\bm\theta}\in C^m((0,1))$. By the implied uniform continuity, the function and its several derivatives can be uniquely extended to continuous functions on $[0,1]$. We also assume that $\bm\theta\mapsto f_{\bm\theta}(x)$ is continuous in $\bm\theta$.
The true regression function $f_0(\cdot)$ does not necessarily lie in $\{f_{\bm\theta}:\bm\theta\in\Theta\}$. We assume that $f_0\in C^m([0,1])$.
Let $\varepsilon_{i}$ are identically and independently distributed with mean zero and finite moment generating function for $i=1,\ldots,n$. Let the common variance be $\sigma^2_0$. We use $N(0,\sigma^2)$ as the working model for the error, which may be different from the true distribution. We treat $\sigma^2$ as an unknown parameter and assign an inverse gamma prior on $\sigma^2$ with shape and scale parameters $a$ and $b$ respectively. Additionally it is assumed that $X_i\stackrel{iid}\sim G$ with density $g$. The approximate solution to \eqref{intro} is given by $f_{\theta,r}$, where $r=r_n$ is the number of grid points, which is chosen so that
\begin{eqnarray}
r_n\gg n^{1/8}.\label{grid}
\end{eqnarray}
Let us denote $\bm Y=(Y_1,\ldots,Y_n)^T$ and $\bm X=(X_1,\ldots,X_n)^T$. The true joint distribution of $(X_i,\varepsilon_i)$ is denoted by $P_0$. Now we describe the two different approaches of inference on $\bm\theta$ used in this paper.
\subsection{Runge-Kutta Sieve Bayesian Method (RKSB)}
For RKSB we denote $\bm\gamma=(\bm\theta,\sigma^2)$. The approximate likelihood of the sample $\{(X_i,Y_{i,}):\,i=1,\ldots,n\}$ is given by $L^*_n({\bm\gamma})=\prod_{i=1}^np_{\bm\gamma,n}(X_i,Y_{i,})$,
where
\begin{eqnarray}
p_{\bm\gamma,n}(X_i, Y_i)=(\sqrt{2\pi}\sigma)^{-1}\exp\{-(2\sigma^2)^{-1}| Y_{i}- f_{\bm\theta,r_n}(X_i)|^2\}g(X_i).\label{approx}
\end{eqnarray}
We also denote
\begin{eqnarray}
p_{\bm\gamma}(X_i, Y_i)=(\sqrt{2\pi}\sigma)^{-1}\exp\{-(2\sigma^2)^{-1}|Y_{i}-f_{\bm\theta}(X_i)|^2\}g(X_i).\label{actual}
\end{eqnarray}
The true parameter $\bm\gamma_0:=(\bm\theta_0,\sigma^2_*)$ is defined as
\begin{eqnarray}
\bm\gamma_0=\arg\max_{\bm\gamma}P_0\log p_{\bm\gamma},\nonumber
\end{eqnarray}
which takes into account the situation when $f_{\bm\theta_0}$ is the true regression function, $\bm\theta_0$ being the true parameter. We denote by $\ell_{\bm\gamma}$ and $\ell_{\bm\gamma,n}$ the log-likelihoods with respect to \eqref{actual} and \eqref{approx} respectively. We make the following assumptions.

(A1) The parameter vector $\bm\gamma_0$ is the unique maximizer of the right hand side above.

(A2) The sub-matrix of the Hessian matrix of $-P_0\log p_{\bm\gamma}$ at $\bm\gamma=\bm\gamma_0$ given by
\begin{eqnarray}
\int_0^1\left(\dot{f}^T_{\bm\theta_0}(t)\dot{f}_{\bm\theta_0}(t)-\frac{\partial}{\partial\bm\theta}\left(\dot{ f}^T_{\bm\theta}(t)\left(f_0(t)-f_{\bm\theta_0}(t)\right)\right)\Big{|}_{\bm\theta=\bm\theta_0}\right)g(t)dt
\end{eqnarray}
is positive definite.

(A3) The prior measure on $\Theta$ has a Lebesgue-density continuous and positive on a neighborhood of $\bm\theta_0$.

(A4) The prior distribution of $\bm\theta$ is independent of that of $\sigma^2$.

By (A1) we get
\begin{eqnarray}
\int_0^1\dot{f}^T_{\bm\theta_0}(t)\left(f_0(t)-f_{\bm\theta_0}(t)\right)g(t)dt&=&\bm 0,\nonumber\\
\sigma^2_*&=&\sigma^2_0+\int_0^1|f_0(t)-f_{\bm\theta_0}(t)|^2g(t)dt.\label{deriv_10}
%\int_0^1\left(\dot{\bm f}^T_{\bm\theta_0}(t)\dot{\bm f}_{\bm\theta_0}(t)-\frac{\partial}{\partial\bm\theta}\left(\dot{\bm f}^T_{\bm\theta}(t)\left(\bm f_0(t)-\bm f_{\bm\theta_0}(t)\right)\right)|_{\bm\theta=\bm\theta_0}\right)dG(t)\,\text{is positive definite}.\label{secderivneg}
\end{eqnarray}
The joint prior measure of $\bm\gamma$ is denoted by $\Pi$ with corresponding density $\pi$. We obtain the posterior of $\bm\gamma$ using the approximate likelihood given by \eqref{approx}.
\begin{remark}
In the RKSB method the space of densities induced by the RK4 numerical solution of the ODEs approaches the space of actual densities as the sample size $n$ goes to infinity. This justifies the use of the term ``sieve'' in ``RKSB''.
\end{remark}

\begin{remark}
The assumptions (A1) and (A2) are necessary to prove the convergence of the Bayes estimator of $\bm\gamma$ to the true value $\bm\gamma_0$. These are usually satisfied in most practical situations for example the Lotka-Volterra equations considered in the simulation study. When the true regression function is the solution of the ODE, the Hessian matrix becomes $\int_0^1\dot{f}^T_{\bm\theta_0}(t)\dot{f}_{\bm\theta_0}(t)g(t)dt$ which is positive definite unless the components of $\dot{f}^T_{\bm\theta_0}$ as is the case in our simulation study.
\end{remark}
\subsection{Runge-Kutta Two-step Bayesian Method (RKTB)}
In the RKTB approach, the proposed model is embedded in nonparametric regression model
\begin{eqnarray}
\bm Y=\bm X_n\bm\beta+\bm\varepsilon,\label{np}
\end{eqnarray}
where $\bm X_n=(\!(N_{j}(X_i))\!)_{1\leq i\leq n,1\leq j\leq k_n+m-1}$, $\{N_{j}(\cdot)\}_{j=1}^{k_n+m-1}$ being the B-spline basis functions of order $m$ with $k_n-1$ interior knots $0<\xi_1<\xi_2<\cdots<\xi_{k_n-1}<1$ chosen to satisfy the pseudo-uniformity criteria:
\begin{eqnarray}
\max_{1\leq i\leq k_n-1}\left|\xi_{i+1}-2\xi_i+\xi_{i-1}\right|=o\left(k^{-1}_n\right),\,\max_{1\leq i\leq k_n-1}\left|\xi_{i}-\xi_{i-1}\right|/\min_{1\leq i\leq k_n-1}\left|\xi_{i}-\xi_{i-1}\right|\leq M\nonumber\\\label{pseudo}
\end{eqnarray}
for some constant $M>0$. Here $\xi_0$ and $\xi_{k_n}$ are defined as $0$ and $1$ respectively. The criteria \eqref{pseudo} is required to apply the asymptotic results obtained in \citet{zhou1998local} where they mention the similar criteria in equation (3) of that paper. We assume for a given $\sigma^2$
\begin{equation}
\bm\beta\sim N_{k_n+m-1}(\bm 0,\sigma^2n^{2}k^{-1}_nI_{k_n+m-1}).\label{prior}
\end{equation}
Simple calculation yields the conditional posterior distribution for $\bm\beta$ given $\sigma^2$ as
\begin{eqnarray}
N_{k_n+m-1}\left({\left({\bm X^T_n\bm X_n}+\frac{k_n}{n^2}I_{k_n+m-1}\right)}^{-1}{\bm X^T_n\bm Y},\sigma^2{\left({\bm X^T_n\bm X_n}+\frac{k_n}{n^2}I_{k_n+m-1}\right)}^{-1}\right).\label{posterior}
\end{eqnarray}
By model \eqref{np}, the expected response at a point $t\in[0,1]$ is given by $\bm\beta^T\bm N(t)$, where $\bm N(\cdot)=(N_{1}(\cdot),\ldots,N_{k_n+m-1}(\cdot))^T$.
Let us denote for a given parameter $\bm\eta$
\begin{eqnarray}
R_{f,n}(\bm\eta)&=&\left\{\int_0^1|f(t)-f_{\eta,r_n}(t)|^2g(t)dt\right\}^{1/2},\nonumber\\
R_{f_0}(\bm\eta)&=&\left\{\int_0^1|f_0(t)-f_{\eta}(t)|^2g(t)dt\right\}^{1/2},\nonumber
\end{eqnarray}
where $f(t)=\bm\beta^T\bm N(t)$. Now we define $\bm\theta=\arg\min_{\bm\eta\in\Theta}R_{ f,n}(\bm\eta)$ and induce posterior distribution on $\Theta$ through the posterior of $\bm\beta$ given by \eqref{posterior}. Also let us define $\bm\theta_0=\arg\min_{\bm\eta\in\Theta}R_{f_0}(\bm\eta)$. Note that this definition of $\bm\theta_0$ takes into account the case when $f_{\bm\theta_0}$ is the true regression function with corresponding true parameter $\bm\theta_0$.
We use the following standard assumptions.

(A5) For all $\epsilon>0$,
\begin{equation}
\inf_{\bm\eta:\|\bm\eta-\bm\theta_0\|\geq\epsilon}R_{f_0}(\bm\eta)>R_{f_0}(\bm\theta_0).\label{assmp}
\end{equation}

(A6) The matrix
\begin{eqnarray}
\bm J_{\bm\theta_0}=-\int_0^1\ddot f_{\bm\theta_0}(t)(f_0(t)- f_{\theta_0}(t))g(t)dt
+\int_0^1\left(\dot f_{\bm\theta_0}(t)\right)^T\left(\dot f_{\bm\theta_0}(t)\right)g(t)dt\nonumber
\end{eqnarray}
is nonsingular.

\begin{remark}
The assumption (A5) implies that $\bm\theta_0$ is a well-separated point of minima of $R_{f_0}(\cdot)$ which is needed to prove the convergence of the posterior distribution of $\bm\theta$ to the true value $\bm\theta_0$. A similar looking assumption appears in the argmax theorem used to show consistency of M-estimators and is a stronger version of the condition of uniqueness of the location of minimum.
\end{remark}

\begin{remark}
The matrix $\bm J_{\bm\theta_0}$ is usually non-singular specially in the case when the true regression function satisfies the ODE since then the expression of $\bm J_{\bm\theta_0}$ becomes $\int_0^1\dot{f}^T_{\bm\theta_0}(t)\dot{f}_{\bm\theta_0}(t)g(t)dt$ which is usually positive definite.
\end{remark}
\section{Main results}
Our main results are given by Theorems $4.1$ and $4.2$.%the next theorem which is a little variation of Theorem 2.1 of \citet{kleijn2012bernstein}
\begin{theorem}
Let the posterior probability measure related to RKSB be denoted by $\Pi_n$. Then posterior of $\bm\gamma$ contracts at $\bm\gamma_0$ at the rate $n^{-1/2}$ and
\begin{eqnarray}
\left\|\Pi_n\left(\sqrt{n}(\bm\gamma-\bm\gamma_0)\in\cdot|\bm X,\bm Y\right)-\bm N(\bm\Delta_{n,\bm\gamma_0},\sigma^2_*\bm V^{-1}_{\bm\gamma_0})\right\|_{TV}=o_{P_0}(1),\nonumber
\end{eqnarray}
where $\bm V_{\bm\gamma_0}=\left({\begin{array}{cc}{\sigma^{-2}_*}{\bm V_{\theta_0}}&\bm 0\\\bm 0&{\sigma^{-4}_*}/2\end{array}}\right)$
with $$\bm V_{\bm\theta_0}=\int_0^1\left(\dot{f}^T_{\bm\theta_0}(t)\dot{f}_{\bm\theta_0}(t)-\frac{\partial}{\partial\bm\theta}\left(\dot{ f}^T_{\bm\theta}(t)\left(f_0(t)-f_{\bm\theta_0}(t)\right)\right)\Big{|}_{\bm\theta=\bm\theta_0}\right)g(t)dt$$ and $\bm\Delta_{n,\bm\gamma_0}=\bm V^{-1}_{\bm\gamma_0}\mathbb{G}_n\dot\ell_{\bm\gamma_0,n}$.
\end{theorem}
Since $\bm\theta$ is a sub-vector of $\bm\gamma$, we get Bernstein-von Mises Theorem for the posterior distribution of $\sqrt{n}(\bm\theta-\bm\theta_0)$, the mean and dispersion matrix of the limiting Gaussian distribution being the corresponding sub-vector and sub-matrix of $\bm\Delta_{n,\bm\gamma_0}$ and $\sigma^2_*\bm V^{-1}_{\bm\gamma_0}$ respectively. We also get the following important corollary.
\begin{corollary}
When the regression model \eqref{prop} is correctly specified and also the error is Gaussian, the Bayes estimator based on $\Pi_n$ is asymptotically efficient.
\end{corollary}
Let us denote $\bm C(t)=\bm J_{\bm\theta_0}^{-1}\left(\dot f_{\bm\theta_0}(t)\right)^T$ and $\bm H^T_n=\int_0^1\bm C(t)\bm N^T(t)g(t)dt$. Note that $\bm C(t)$ is a $p$-component vector. Also, we denote the posterior probability measure of RKTB by $\Pi_n^*$. Now we have the following result.

\begin{theorem}\label{bvm_RKTB}
Let
\begin{eqnarray}
\bm\mu_n&=&\sqrt{n}\bm H_{n}^T\left(\bm X^T_n\bm X_n\right)^{-1}\bm X^T_n\bm Y-\sqrt{n}\int_0^1\bm C(t)f_0(t)g(t)dt,\nonumber\\
\bm\Sigma_n&=&n\bm H_{n}^T\left(\bm X^T_n\bm X_n\right)^{-1}\bm H_{n}\nonumber
\end{eqnarray}
and $\bm B=\left(\!\left(\left<C_{k}(\cdot),C_{k'}(\cdot)\right>_g\right)\!\right)_{k,k'=1,\ldots,p}$. 
If $\bm B$ is non-singular, then for $m\geq3$ and $n^{1/(2m)}\ll k_n\ll n^{1/2}$,
\begin{eqnarray}
\|\Pi^*_n\left(\sqrt{n}(\bm\theta-\bm\theta_0)\in\cdot|Y\right)-N\left(\bm\mu_n,\sigma^2_0\bm\Sigma_n\right)\|_{TV}=o_{P_0}(1).\label{thm2}
\end{eqnarray}
\end{theorem}
\begin{remark}
{It will be proved later in Lemma 10 that both $\bm\mu_n$ and $\bm\Sigma_n$ are stochastically bounded. Hence, with high true probability the posterior distribution of $(\bm\theta-\bm\theta_0)$ contracts at $\bm0$ at $n^{-1/2}$ rate}.
\end{remark}
We also get the following important corollary.
\begin{corollary}
When the regression model \eqref{prop} is correctly specified and the true distribution of error is Gaussian, the Bayes estimator based on $\Pi^*_n$ is asymptotically efficient.
\end{corollary}
\begin{remark}
The Bayesian two-step approach \citep{bhaumik2014bayesiant} considers the distance between the derivative of the function in the nonparametric model and the derivative given by the ODE. On the other hand RKTB approach deals directly with the distance between the function in the nonparametric model and the parametric nonlinear regression model through the RK4 approximate solution of the ODE. Direct distance in the latter approach produces the efficient linearization giving rise to efficient concentration of the posterior distribution which can be traced back to efficiency properties of minimum distance estimation methods depending on the nature of the distance.
\end{remark}
\begin{remark}
RKSB is the Bayesian analog of estimating $\bm\theta$ as $$\hat{\bm\theta}=\arg\min_{\bm\eta\in\Theta}\sum_{i=1}^n(Y_i-f_{\bm\eta,r_n})^2.$$ Similarly, RKTB is the Bayesian analog of $\hat{\bm\theta}=\arg\min_{\bm\eta\in\Theta}\int_0^1(\hat{f}(t)-f_{\bm\eta,r_n})^2g(t)dt$, where $\hat{f}(\cdot)$ stands for the nonparametric estimate of $f$ based on B-splines. Arguments similar to ours should be able to establish analogous convergence results for these estimators.
\end{remark}
\section{Simulation Study}
We consider the Lotka-Volterra equations to study the posterior distribution of $\bm\theta$. We consider two cases. In case $1$ the true regression function belongs to the solution set and in case 2 it does not. Thus we have $p=4,d=2$ and the ODE's are given by
\begin{eqnarray}
F_1(t, \bm f_{\bm\theta}(t), \bm\theta)&=&\theta_1f_{1\bm\theta}(t)-\theta_2f_{1\bm\theta}(t)f_{2\bm\theta}(t),\nonumber\\
F_2(t, \bm f_{\bm\theta}(t), \bm\theta)&=&-\theta_3f_{2\bm\theta}(t)+\theta_4f_{1\bm\theta}(t)f_{2\bm\theta}(t),\,\,t\in[0,1],\nonumber
\end{eqnarray}
with initial condition $f_{1\bm\theta}(0)=1,f_{2\bm\theta}(0)=0.5$. The above system is not analytically solvable.

Case 1 (well-specified case): The true regression function is $\bm f_0(t)=(f_{1\bm\theta_0}(t),f_{2\bm\theta_0}(t))^T$ where $\bm\theta_0=(10,10,10,10)^T.$

Case 2 (misspecified case): The true regression function is $\bm f_0(t)=(f_{1\bm\tau_0}(t)+\frac{t^2+t-c_1}{6},f_{2\bm\tau_0}(t)+\frac{t^2+t-c_2}{6})^T$ where $\bm\tau_0=(10,10,10,10)^T$ and $c_1$ and $c_2$ are chosen so that $$\int_0^1f_{1\bm\tau_0}(t)\left(t^2+t-c_1\right)=\int_0^1f_{2\bm\tau_0}(t)\left(t^2+t-c_2\right)=0.$$

For a sample of size $n$, the $X_i$'s are drawn from $\mathrm{Uniform}(0,1)$ distribution for $i=1,\ldots,n$. Samples of sizes 100 and 500 are considered. We simulate 900 replications for each case. The output are displayed in Table 1 and Table 2 respectively. Under each replication a sample of size 1000 is drawn from the posterior distribution of $\bm\theta$ using RKSB, RKTB and Bayesian two-step \citep{bhaumik2014bayesiant} methods and then 95\% equal tailed credible intervals are obtained. For case 2 we do not consider the Bayesian two-step method since there is no existing result on asymptotic efficiency under misspecification of the regression function and hence it is not comparable with the numerical solution based methods. The Bayesian two-step method is abbreviated as TS in Table 1. We calculate the coverage and the average length of the corresponding credible intervals over these $900$ replications. The estimated standard errors of the interval length and coverage are given inside the parentheses in the tables. %We also consider 1000 replications to construct the $95\%$ equal tailed confidence interval based on asymptotic normality as obtained from the estimation method introduced by \citet{varah1982spline} and modified and studied by \citet{brunel2008parameter}. We abbreviate this method by ``VB" in tables. The estimated standard errors of the interval length and coverage are given inside the parentheses in the tables.

The true distribution of error is taken $N(0,(0.1)^2)$. We put an inverse gamma prior on $\sigma^2$ with shape and scale parameters being $30$ and $5$ respectively. For RKSB the prior for each $\theta_j$ is chosen as independent Gaussian distribution with mean $6$ and variance $16$ for $j=1,\ldots,4$. We take $n$ grid points to obtain the numerical solution of the ODE by RK$4$ for a sample of size $n$. According to the requirements of Theorem \ref{bvm_RKTB} of this paper and Theorem 1 of \citet{bhaumik2014bayesiant}, we take $m=3$ and $m=5$ for RKTB and Bayesian two-step method respectively. In both cases we choose $k_n-1$ equispaced interior knots $\frac{1}{k_n},\frac{2}{k_n},\ldots,\frac{k_n-1}{k_n}$. This specific choice of knots satisfies the pseudo-uniformity criteria \eqref{pseudo} with $M=1$. Looking at the order of $k_n$ suggested by Theorem $4.2$, $k_n$ is chosen in the order of $n^{1/5}$ giving the values of $k_n$ as $13$ and $18$ for $n=100$ and $n=500$ respectively in RKTB. In Bayesian two-step method the values of $k_n$ are $17$ and $20$ for $n=100$ and $n=500$ respectively by choosing $k_n$ in the order of $n^{1/9}$ following the suggestion given in Theorem 1 of \citet{bhaumik2014bayesiant}. In all the cases the constant multiplier to the chosen asymptotic order is selected through cross-validation.

We separately analyze the output given in Table 1 since it deals with asymptotic efficiency. Not surprisingly the first two methods perform much better compared to the third one because of asymptotic efficiency obtained from Corollaries $1$ and $2$ respectively. For RKSB a single replication took about one hour and four hours for samples of sizes $100$ and $500$ respectively. For RKTB these times are around one hour and two and half hours respectively. In Bayesian two-step method each replication took about one and two minutes for $n=100$ and $500$ respectively. Thus from the computational point of view Bayesian two-step method is preferable than the numerical solution based approaches.
\begin{table}[h]
\centering
\caption{\textit{Coverages and average lengths of the Bayesian credible intervals for the three methods in case of well-specified regression model}}
\begin{tabular}{|c|c|cc|cc|cc|}
\hline
%\multirow{2}{*}{$n$}&&\multicolumn{4}{|c|}{$N(0,(0.2)^2)$}&\multicolumn{4}{|c|}{scaled $t_6$}\\
%\cline{3-10}
$n$&&\multicolumn{2}{|c|}{RKSB}&\multicolumn{2}{|c|}{RKTB}&\multicolumn{2}{|c|}{TS}\\
\hline
&&coverage&length&coverage&length&coverage&length\\
&&(se)&(se)&(se)&(se)&(se)&(se)\\
100&$\theta_1$&100.0&2.25&100.0&2.17&100.0&6.93\\
&&(0.00)&(0.29)&(0.00)&(0.65)&(0.00)&(4.95)\\
&$\theta_2$&100.0&2.57&100.0&2.48&100.0&6.67\\
&&(0.00)&(0.33)&(0.00)&(0.74)&(0.00)&(4.90)\\
&$\theta_3$&99.9&2.50&100.0&2.44&100.0&7.12\\
&&(0.00)&(0.34)&(0.00)&(1.44)&(0.00)&(4.92)\\
&$\theta_4$&100.0&2.27&100.0&2.20&100.0&6.59\\
&&(0.00)&(0.32)&(0.00)&(1.19)&(0.00)&(4.77)\\
\cline{2-8}
500&$\theta_1$&100.0&0.75&99.4&0.56&99.2&1.09\\
&&(0.00)&(0.06)&(0.00)&(0.02)&(0.00)&(0.05)\\
&$\theta_2$&100.0&0.85&99.4&0.64&98.8&1.02\\
&&(0.00)&(0.07)&(0.00)&(0.02)&(0.00)&(0.05)\\
&$\theta_3$&100.0&0.82&99.3&0.61&99.0&1.16\\
&&(0.00)&(0.07)&(0.00)&(0.02)&(0.00)&(0.05)\\
&$\theta_4$&99.9&0.74&99.3&0.56&99.0&1.04\\
&&(0.00)&(0.06)&(0.00)&(0.02)&(0.00)&(0.05)\\
%\cline{2-10}
\hline
\end{tabular}
\end{table}

\begin{table}[h]
\centering
\caption{\textit{Coverages and average lengths of the Bayesian credible intervals for RKSB and RKTB in case of misspecified regression model}}
\begin{tabular}{|c|c|cc|cc|}
\hline
%\multirow{2}{*}{$n$}&&\multicolumn{4}{|c|}{$N(0,(0.2)^2)$}&\multicolumn{4}{|c|}{scaled $t_6$}\\
%\cline{3-10}
$n$&&\multicolumn{2}{|c|}{RKSB}&\multicolumn{2}{|c|}{RKTB}\\
\hline
&&coverage&length&coverage&length\\
&&(se)&(se)&(se)&(se)\\
100&$\theta_1$&99.4&2.32&100.0&2.21\\
&&(0.01)&(0.31)&(0.00)&(0.83)\\
&$\theta_2$&99.1&2.64&100.0&2.46\\
&&(0.01)&(0.36)&(0.00)&(0.79)\\
&$\theta_3$&99.4&2.78&100.0&2.7\\
&&(0.01)&(0.46)&(0.00)&(1.73)\\
&$\theta_4$&99.3&2.5&100.0&2.38\\
&&(0.01)&(0.43)&(0.00)&(1.43)\\
\cline{2-6}
500&$\theta_1$&98.8&0.89&99.3&0.56\\
&&(0.00)&(0.07)&(0.00)&(0.02)\\
&$\theta_2$&98.9&1&99.5&0.62\\
&&(0.00)&(0.13)&(0.00)&(0.02)\\
&$\theta_3$&99.0&1.05&99.2&0.66\\
&&(0.00)&(0.09)&(0.00)&(0.03)\\
&$\theta_4$&99.2&0.94&99.1&0.59\\
&&(0.00)&(0.08)&(0.00)&(0.02)\\
%\cline{2-10}
\hline
\end{tabular}
\end{table}

\section{Proofs}
We use the operators $\E_0(\cdot)$ and $\Var_0(\cdot)$ to denote expectation and variance with respect to $P_0$.
%%%%%%%%%%Proof of Theorem 4.1

\begin{proof}[Proof of Theorem 4.1]
From Lemma $1$ below we know that there exists a compact subset $U$ of $(0,\infty)$ such that $\Pi_n(\sigma^2\in U|\bm X,\bm Y)\stackrel{P_0}\rightarrow1$. Let $\Pi_{U,n}(\cdot|\bm X,\bm Y)$ be the posterior distribution conditioned on $\sigma^2\in U$. By Theorem 2.1 of \citet{kleijn2012bernstein} if we can ensure that there exist stochastically bounded random variables $\bm\Delta_{n,\gamma_0}$ and a positive definite matrix $\bm V_{\bm\gamma_0}$ such that for every compact set $K\subset\mathbb{R}^{p+1}$,
\begin{eqnarray}
\sup_{h\in K}\left|\log\frac{p^{(n)}_{\bm\gamma_0+\bm h/\sqrt{n},n}}{p^{(n)}_{\bm\gamma_0,n}}(\bm X,Y)-\bm h^T\bm V_{\bm\gamma_0}\bm\Delta_{n,\bm\gamma_0}+\frac{1}{2}\bm h^T\bm V_{\bm\gamma_0}\bm h\right|\rightarrow0,\label{cond1}
\end{eqnarray}
in (outer) $P_0^{(n)}$ probability and that for every sequence of constants $M_n\rightarrow\infty$, we have
\begin{eqnarray}
P^{(n)}_0\Pi_{U,n}\left(\sqrt{n}\|\bm\gamma-\bm\gamma_0\|>M_n|\bm X,Y\right)\rightarrow0,\label{cond2}
\end{eqnarray}
then
\begin{eqnarray}
\left\|\Pi_{U,n}\left(\sqrt{n}(\bm\gamma-\bm\gamma_0)\in\cdot|\bm X,Y\right)-\bm N(\bm\Delta_{n,\bm\gamma_0},\bm V^{-1}_{\bm\gamma_0})\right\|_{TV}=o_{P_0}(1)\nonumber.
\end{eqnarray}
We show that the conditions \eqref{cond1} and \eqref{cond2} hold in Lemmas $1$ to $5$. Lemma $2$ gives that $\bm V_{\bm\gamma_0}=\left({\begin{array}{cc}{\sigma^{-2}_*}{\bm V_{\theta_0}}&\bm 0\\\bm 0&{\sigma^{-4}_*}/2\end{array}}\right)$ with $$\bm V_{\bm\theta_0}=\int_0^1\left(\dot{f}^T_{\bm\theta_0}(t)\dot{f}_{\bm\theta_0}(t)-\frac{\partial}{\partial\bm\theta}\left(\dot{ f}^T_{\bm\theta}(t)\left(f_0(t)-f_{\bm\theta_0}(t)\right)\right)\Big{|}_{\bm\theta=\bm\theta_0}\right)g(t)dt$$ and $\bm\Delta_{n,\bm\gamma_0}=\bm V^{-1}_{\bm\gamma_0}\mathbb{G}_n\dot\ell_{\bm\gamma_0,n}$. Since $\|\Pi_{n}-\Pi_{U,n}\|_{TV}=o_{P_0}(1)$, we get
\begin{eqnarray}
\left\|\Pi_n\left(\sqrt{n}(\bm\gamma-\bm\gamma_0)\in\cdot|\bm X,Y\right)-\bm N(\bm\Delta_{n,\bm\gamma_0},\bm V^{-1}_{\bm\gamma_0})\right\|_{TV}=o_{P_0}(1)\nonumber.
\end{eqnarray}
Hence, we get the desired result.
\end{proof}

%%%%%%%%% Proof of Corollary 1

\begin{proof}[Proof of Corollary 1]
The log-likelihood of the correctly specified model with Gaussian error is given by
\begin{eqnarray}
\ell_{\bm\gamma_0}(X,Y)&=&-\log\sigma_0-\frac{1}{2\sigma^2_0}|Y-f_{\bm\theta_0}(X)|^2+\log g(X)\nonumber.
\end{eqnarray}
Thus $\frac{\partial}{\partial\bm\theta_0}\ell_{\bm\gamma_0}(X,Y)=\sigma^{-2}_0\left(\dot{f}_{\bm\theta_0}(X)\right)^T\left(Y-f_{\bm\theta_0}(X)\right)$ and $\frac{\partial}{\partial\sigma^2_0}\ell_{\bm\gamma_0}(X,Y)=-\frac{1}{2\sigma^2_0}+\frac{1}{2\sigma^4_0}|Y-f_{\bm\theta_0}(X)|^2$. Hence, the Fisher information is given by
\begin{eqnarray}
\bm I(\bm\gamma_0)=\left({\begin{array}{cc}{{\sigma^{-2}_0}\int_0^1\dot{f}^T_{\bm\theta_0}(t)\dot{f}_{\bm\theta_0}(t)g(t)dt}&\bm 0\\\bm 0&{\sigma^{-4}_0}/2\end{array}}\right).\nonumber
\end{eqnarray}
Looking at the form of $\bm V_{\bm\gamma_0}$ in Theorem 4.1, we get $\bm V^{-1}_{\bm\gamma_0}=\left(\bm I(\bm\gamma_0)\right)^{-1}$ if the regression function is correctly specified and the true error distribution is $N(0,\sigma^2_0)$.
\end{proof}

%%%%%%%%%%Proof of Theorem 4.2

\begin{proof}[Proof of Theorem 4.2]
We have for $f(\cdot)=\bm\beta^T\bm N(\cdot)$
\begin{eqnarray}
\bm J_{\bm\theta_0}^{-1}\bm\Gamma(f)&=&\int_0^1\bm C(t)\bm\beta^T\bm N(t)g(t)dt=\bm H_n^T\bm\beta,\label{linearize}
\end{eqnarray}
where $\bm H_{n}^T=\int_0^1\bm C(t)\bm N^T(t)g(t)dt$ which is a matrix of order $p\times (k_n+m-1)$. Consequently, the asymptotic variance of the conditional posterior distribution of $\bm H^T_n\bm\beta$ is $\sigma^2\bm H^T_{n}\left(\bm X^T_n\bm X_n+\frac{k_n}{n^2}\bm I\right)^{-1}\bm H_{n}$.
By Lemma $9$ and the posterior consistency of the $\sigma^2$ given by Lemma $11$, it suffices to show that for any neighborhood $\mathscr{N}$ of $\sigma^2_0$,
\begin{eqnarray}
\sup_{\sigma^2\in\mathscr{N}}\left\|\Pi^*_n\left(\sqrt{n}\bm H_{n}^T\bm\beta-\sqrt{n}\bm J^{-1}_{\bm\theta_0}\bm\Gamma(f_0)\in\cdot|\bm X,\bm Y,\sigma^2\right)-{N}(\bm\mu_n,\sigma^2\bm\Sigma_n)\right\|_{TV}
=o_{P_0}(1).\label{thm2_1}
\end{eqnarray}
Note that $\Pi(\mathscr{N}^c|\bm X,\bm Y)=o_{P_0}(1)$. It is straightforward to verify that the Kullback-Leibler divergence between $N\left(\left(\bm X^T_n\bm X_n\right)^{-1}\bm X^T_n\bm Y,\sigma^2\left(\bm X^T_n\bm X_n\right)^{-1}\right)$ and the distribution given by \eqref{posterior} converges in $P_0$-probability to zero uniformly over $\sigma^2\in\mathscr{N}$ and hence, so is the total variation distance. By linear transformation \eqref{thm2_1} follows. Note that
\begin{eqnarray}
\lefteqn{\sup_{B\in\mathscr{R}^p}\left|\Pi(\sqrt{n}(\bm\theta-\bm\theta_0)\in B|\bm X,\bm Y)-\Phi(B;\bm\mu_n,\sigma_0^2\bm\Sigma_n)\right|}\nonumber\\
&\leq&\int\sup_{B\in\mathscr{R}^p}\left|\Pi(\sqrt{n}(\bm\theta-\bm\theta_0)\in B|\bm X,\bm Y,\sigma^2)-\Phi(B;\bm\mu_n,\sigma^2\bm\Sigma_n)\right|d\Pi(\sigma^2|\bm X,\bm Y)\nonumber\\
&&+\int\sup_{B\in\mathscr{R}^p}\left|\Phi(B;\bm\mu_n,\sigma^2\bm\Sigma_n)-\Phi(B;\bm\mu_n,\sigma_0^2\bm\Sigma_n)\right|d\Pi(\sigma^2|\bm X,\bm Y)\nonumber\\
&\leq&\sup_{\sigma^2\in\mathscr{N}}\sup_{B\in\mathscr{R}^p}\left|\Pi(\sqrt{n}(\bm\theta-\bm\theta_0)\in B|\bm X,\bm Y,\sigma^2)-\Phi(B;\bm\mu_n,\sigma^2\bm\Sigma_n)\right|\nonumber\\
&&+\sup_{\sigma^2\in\mathscr{N},B\in\mathscr{R}^p}\left|\Phi(B;\bm\mu_n,\sigma^2\bm\Sigma_n)-\Phi(B;\bm\mu_n,\sigma_0^2\bm\Sigma_n)\right|+2\Pi(\mathscr{N}^c|\bm X,\bm Y).\nonumber
\end{eqnarray}
Using the fact that $\bm\Sigma_n$ is stochastically bounded given by Lemma $10$, the total variation distance between the two normal distributions appearing in the second term of the above display is bounded by a constant multiple of $|\sigma^2-\sigma^2_0|$, and hence can be made arbitrarily small by choosing $\mathscr{N}$ accordingly. The first term converges in probability to zero by \eqref{thm2_1}. The third term converges in probability to zero by the posterior consistency.
\end{proof}

%%%%%%%%%%Proof of Corollary 2

\begin{proof}[Proof of Corollary 2]
The log-likelihood of the correctly specified model is given by
\begin{eqnarray}
\ell_{\bm\theta_0}(X,Y)&=&-\log\sigma_0-\frac{1}{2\sigma^2_0}|Y-f_{\bm\theta_0}(X)|^2+\log g(X)\nonumber.
\end{eqnarray}
Thus $\dot\ell_{\bm\theta_0}(X,Y)=-\sigma^{-2}_0\left(\dot{f}_{\bm\theta_0}(X)\right)^T\left(Y-f_{\bm\theta_0}(X)\right)$ and the Fisher information is given by
$\bm I(\bm\theta_0)=\sigma^{-2}_0\int_0^1\left(\dot{f}_{\bm\theta_0}(X)\right)^T\dot{f}_{\bm\theta_0}(t)g(t)dt$.
In the proof of Lemma $10$ we obtained that $\sigma^2_0\bm\Sigma_n\stackrel{P_0}\rightarrow\sigma^2_0\bm J^{-1}_{\bm\theta_0}\int_0^1\left(\dot f_{\bm\theta_0}(t)\right)^T\dot f_{\bm\theta_0}(t)g(t)dt\,\bm J^{-1}_{\bm\theta_0}$. This limit is equal to $\left(\bm I({\bm\theta_0})\right)^{-1}$ under the correct specification of the regression function as well as the likelihood.
\end{proof}

\section{Proofs of technical lemmas}
The first five lemmas in this section are related to RKSB. The rest are for RKTB. The first lemma shows that the posterior of $\sigma^2$ lies inside a compact set with high probability.

%%%%%%%%%%Proof of Lemma 1

\begin{lemma}
There exists a compact set $U$ independent of $\bm\theta$ and $n$ such that $\Pi_n(\sigma^2\in U|\bm X,\bm Y)\stackrel{P_0}\rightarrow1$.
\end{lemma}
\begin{proof}
Given $\bm\theta$, the conditional posterior of $\sigma^2$ is an inverse gamma distribution with shape and scale parameters $n/2+a$ and ${2}^{-1}\sum_{i=1}^n(Y_i-f_{\bm\theta}(X_i))^2+b$ respectively. Clearly $\E(\sigma^2|\bm X,\bm Y,\bm\theta)=n^{-1}\sum_{i=1}^n(Y_i-f_{\bm\theta}(X_i))^2+o(1)$ a.s. Hence, it is easy to show using the weak law of large numbers that the mean of the conditional posterior of $\sigma^2$ converges in $P_0$-probability to $\sigma^2_{\bm\theta}:=\sigma^2_0+\int_0^1\left(f_0(t)-f_{\bm\theta}(t)\right)^2g(t)dt$. Then it follows that for any $\epsilon>0$, $\Pi_n\left(\sigma^2\in[\sigma^2_{\bm\theta}-\epsilon,\sigma^2_{\bm\theta}+\epsilon]|\bm X,\bm Y,\bm\theta\right)$ converges in $P_0$-probability to $1$.
Since $\Theta$ is compact and $\sigma^2_{\bm\theta}$ is continuous in $\bm\theta$, there exists a compact set $U$ such that $U\supseteq [\sigma^2_{\bm\theta}-\epsilon,\sigma^2_{\bm\theta}+\epsilon]$ for all $\bm\theta$. Now $\Pi_n\left(\sigma^2\notin U|\bm X,\bm Y\right)$ is bounded above by
\begin{eqnarray}
\lefteqn{\int_{\Theta}\Pi_n\left(|\sigma^2-\sigma^2_{\bm\theta}|>\epsilon|\bm X,\bm Y,\bm\theta\right)d\Pi_n(\bm\theta|\bm X,\bm Y)}\nonumber\\
&\leq&\epsilon^{-2}\int_{\Theta}\left(\left(\E(\sigma^2|\bm X,\bm Y,\bm\theta)-\sigma^2_{\bm\theta}\right)^2+\Var(\sigma^2|\bm X,\bm Y,\bm\theta)\right)d\Pi_n(\bm\theta|\bm X,\bm Y)\nonumber.
\end{eqnarray}
It suffices to prove that $$\sup_{\bm\theta\in\Theta}\left|\E(\sigma^2|\bm X,\bm Y,\bm\theta)-\sigma^2_{\bm\theta}\right|=o_{P_0}(1)\, \text{and}\, \sup_{\bm\theta\in\Theta}\Var(\sigma^2|\bm X,\bm Y,\bm\theta)=o_{P_0}(1).$$
Using the facts that $\bm\theta\mapsto f_{\bm\theta}(x)$ is Lipschitz continuous and other smoothness criteria of $f_{\bm\theta}(x)$ and $f_0(x)$ and applying Theorem $19.4$ and example $19.7$ of \citet{van_der_Vaart_1998}, it follows that $\{(Y-f_{\bm\theta}(X))^2:\bm\theta\in\Theta\}$ is $P_0$-Glivenko-Cantelli and hence
\begin{eqnarray}
\sup_{\bm\theta\in\Theta}\left|\E(\sigma^2|\bm X,\bm Y,\bm\theta)-\E_0\left(\E(\sigma^2|\bm X,\bm Y,\bm\theta)\right)\right|=o_{P_0}(1).\nonumber
\end{eqnarray}
Also, it can be easily shown that the quantity $\sup_{\bm\theta\in\Theta}\left|\E_0\left(\E(\sigma^2|\bm X,\bm Y,\bm\theta)\right)-\sigma^2_{\bm\theta}\right|\rightarrow0$ as $n\rightarrow\infty$ since
\begin{eqnarray}
\E_0\left(\E(\sigma^2|\bm X,\bm Y,\bm\theta)\right)&=&\sigma^2_0+\int_0^1\left(f_0(t)-f_{\bm\theta}(t)\right)^2g(t)dt\nonumber\\
&&-\frac{2(a-1)\left(\sigma^2_0+\int_0^1\left(f_0(t)-f_{\bm\theta}(t)\right)^2g(t)dt\right)}{n+2a-2}+\frac{2b}{n+2a-2}\nonumber
\end{eqnarray}
and the parameter space $\Theta$ is compact and the mapping $\bm\theta\mapsto f_{\bm\theta}(\cdot)$ is continuous. This gives the first assertion. To see the second assertion, observe that $\Var(\sigma^2|\bm X,\bm Y,\bm\theta)=O(n^{-1})$ a.s. by the previous assertion and the fact that the conditional posterior of $\sigma^2$ given $\bm\theta$ is inverse gamma.
\end{proof}
In view of the previous lemma, we choose the parameter space for $\bm\gamma$ to be $\Theta\times U$ from now onwards. We show that the condition \eqref{cond1} holds by the following Lemma.

%%%%%%%%%%Proof of Lemma 2

\begin{lemma}
For the model induced by Runge-Kutta method as described in Section 3, we have
\begin{eqnarray}
\sup_{\bm h\in K}\left|\log\frac{\prod_{i=1}^np_{\bm\gamma_0+\bm h/\sqrt{n},n}(X_i,Y_i)}{\prod_{i=1}^np_{\bm\gamma_0,n}(X_i,Y_i)}-\bm h^T\bm V_{\bm\gamma_0}\bm\Delta_{n,\gamma_0}+\frac{1}{2}\bm h^T\bm V_{\gamma_0}\bm h\right|\rightarrow0,\nonumber
\end{eqnarray}
in (outer) $P^{(n)}_0$-probability for every compact set $\bm K\subset\mathbb{R}^{p+1}$, where
$$\bm\Delta_{n,\bm\gamma_0}=\bm V^{-1}_{\bm\gamma_0}\mathbb{G}_n\dot\ell_{\bm\gamma_0,n}$$ and $\bm V_{\bm\gamma_0}=\left({\begin{array}{cc}{\sigma^{-2}_*}{\bm V_{\theta_0}}&\bm 0\\\bm 0&{\sigma^{-4}_*}/2\end{array}}\right)$ with $$\bm V_{\bm\theta_0}=\int_0^1\left(\dot{f}^T_{\bm\theta_0}(t)\dot{f}_{\bm\theta_0}(t)-\frac{\partial}{\partial\bm\theta}\left(\dot{ f}^T_{\bm\theta}(t)\left(f_0(t)-f_{\bm\theta_0}(t)\right)\right)\Big{|}_{\bm\theta=\bm\theta_0}\right)g(t)dt.$$
\end{lemma}
\begin{proof}
Let $G$ be an open neighborhood containing $\bm\gamma_0$. For $\bm\gamma_1,\bm\gamma_2\in G$, we have
\begin{eqnarray}
\left|\log(p_{\bm\gamma_1}(X_1,Y_1)/p_{\bm\gamma_2}(X_1,Y_1))\right|\leq m(X_1,Y_1)\|\gamma_1-\gamma_2\|,\nonumber
\end{eqnarray}
where $m(X_1,Y_1)$ is
\begin{eqnarray}
\sup\left\{\frac{|Y_1-f_{\bm\theta}(X_1)|}{\sigma^2}\|\dot{f}_{\bm\theta}(X_1)\|+\frac{(Y_1-f_{\bm\theta}(X_1))^2}{2\sigma^{4}}+\frac{1}{2\sigma^2}: (\bm\theta,\sigma^2)\in G\right\},\nonumber
\end{eqnarray}
which is square integrable. Therefore, by Lemma 19.31 of \citet{van_der_Vaart_1998}, for any sequence $\{\bm h_n\}$ bounded in $P_0$-probability, $$\mathbb{G}_n\left(\sqrt{n}(\ell_{\bm\gamma_0+(\bm h_n/\sqrt{n})}-\ell_{\bm\gamma_0})-\bm h^T_n\dot{\bm\ell}_{\bm\gamma_0}\right)=o_{P_0}(1).$$ Using the laws of large numbers and \eqref{rk4}, we find that $$\mathbb{G}_n\left(\sqrt{n}(\ell_{\bm\gamma_0+(\bm h_n/\sqrt{n})}-\ell_{\bm\gamma_0})-\bm h^T_n\dot{\bm\ell}_{\bm\gamma_0}\right)-\mathbb{G}_n\left(\sqrt{n}(\ell_{\bm\gamma_0+(\bm h_n/\sqrt{n}),n}-\ell_{\bm\gamma_0,n})-\bm h^T_n\dot{\bm\ell}_{\bm\gamma_0,n}\right)$$ is $O_{P_0}(\sqrt{n}r^{-4}_n)$ which is $o_{P_0}(1)$ by the condition \eqref{grid} on $r_n$. Hence, $$\mathbb{G}_n\left(\sqrt{n}(\ell_{\bm\gamma_0+(\bm h_n/\sqrt{n}),n}-\ell_{\bm\gamma_0,n})-\bm h^T_n\dot{\bm\ell}_{\bm\gamma_0,n}\right)=o_{P_0}(1).$$
We note that
\begin{eqnarray}
\lefteqn{-P_0\log(p_{\bm\gamma,n}/p_{\bm\gamma_0,n})}\nonumber\\
&=&\log\sigma-\log\sigma_*+\frac{1}{2\sigma^2}\left[\sigma^2_0+\int_0^1|f_0(t)-f_{\bm\theta,r_n}(t)|^2g(t)dt\right]\nonumber\\
&&-\frac{1}{2\sigma^2_*}\left[\sigma^2_0+\int_0^1|f_0(t)-f_{\bm\theta_0,r_n}(t)|^2g(t)dt\right]\nonumber\\
&=&\log\sigma-\log\sigma_*+\left(\frac{1}{2\sigma^2}-\frac{1}{2\sigma^2_*}\right)\left[\sigma^2_0+\int_0^1|f_0(t)- f_{\bm\theta,r_n}(t)|^2g(t)dt\right]\nonumber\\
&&+\frac{1}{2\sigma^2_*}\left[2\int_0^1\left(f_0(t)-f_{\bm\theta_0,r_n}(t)\right)\left(f_{\bm\theta_0,r_n}(t)-f_{\bm\theta,r_n}(t)\right)g(t)dt\right.\nonumber\\
&&\left.+\int_0^1|f_{\bm\theta_0,r_n}(t)-f_{\bm\theta,r_n}(t)|^2g(t)dt\right].\label{lemma1_1}
\end{eqnarray}
Using \eqref{deriv_10}, the last term inside the third bracket in \eqref{lemma1_1} can be expanded as
\begin{eqnarray}
%\lefteqn{2\int_0^1\left(f_0(t)-f_{\bm\theta_0,n}(t)\right)\left(f_{\bm\theta_0,n}(t)-f_{\bm\theta,n}(t)\right)dG(t)}\nonumber\\
%&&+(\bm\theta-\bm\theta_0)^T\int_0^1\dot{f}^T_{\bm\theta_0}(t)\dot{ f}_{\bm\theta_0}(t)dG(t)(\bm\theta-\bm\theta_0)+o\left(\|\bm\theta-\bm\theta_0\|^2\right).\nonumber\\
(\bm\theta-\bm\theta_0)^T\bm V_{\bm\theta_0}(\bm\theta-\bm\theta_0)+O\left(r^{-4}_n\|\bm\theta-\bm\theta_0\|\right)+o\left(\|\bm\theta-\bm\theta_0\|^2\right)\nonumber,
\end{eqnarray}
where $\bm V_{\bm\theta_0}=\int_0^1\left(\dot{f}^T_{\bm\theta_0}(t)\dot{f}_{\bm\theta_0}(t)-\frac{\partial}{\partial\bm\theta}\left(\dot{ f}^T_{\bm\theta}(t)\left(f_0(t)-f_{\bm\theta_0}(t)\right)\right)\Big{|}_{\bm\theta=\bm\theta_0}\right)g(t)dt$. Also, writing $\sigma^2_*=\sigma^2_0+\int_0^1|f_0(t)-f_{\bm\theta_0}(t)|^2g(t)dt$ and using \eqref{deriv_10}, the first term in \eqref{lemma1_1} is given by
\begin{eqnarray}
%\lefteqn{\log\sigma-\log\sigma_*+\left(\frac{1}{2\sigma^2}-\frac{1}{2\sigma^2_*}\right)\left[\sigma^2_*-\int_0^1|f_0(t)- f_{\bm\theta_0}(t)|^2dG(t)\right.}\nonumber\\
%&&\left.+\int_0^1|f_0(t)-f_{\bm\theta,n}(t)|^2dG(t)\right]\nonumber\\
&&-\frac{1}{2}\log\left(\frac{\sigma^2_*}{\sigma^2}-1+1\right)+\frac{1}{2}\left(\frac{\sigma^2_*}{\sigma^2}-1\right)+O\left(r^{-4}_n\|\bm\gamma-\bm\gamma_0\|\right)+o(\|\bm\gamma-\bm\gamma_0\|^2)\nonumber\\
&&=\frac{(\sigma^2-\sigma^2_*)^2}{4\sigma^4_*}+O\left(r^{-4}_n\|\bm\gamma-\bm\gamma_0\|\right)+o(\|\bm\gamma-\bm\gamma_0\|^2).\nonumber
\end{eqnarray}
Hence,
\begin{eqnarray}
P_0\log\frac{p_{\bm\gamma_0+\bm h_n/\sqrt{n},n}}{p_{\bm\gamma_0,n}}+\frac{1}{2n}\bm h_n^T\bm V_{\bm\gamma_0}\bm h_n=o(n^{-1}).\label{lemma1_2}
\end{eqnarray}
We have already shown that
\begin{eqnarray}
n\mathbb{P}_n\log\frac{p_{\bm\gamma_0+\bm h_n/\sqrt{n},n}}{p_{\bm\gamma_0,n}}-\mathbb{G}_n\bm h^T_n\dot{\bm\ell}_{\bm\gamma_0,n}-nP_0\log\frac{p_{\bm\gamma_0+\bm h_n/\sqrt{n},n}}{p_{\bm\gamma_0,n}}=o_{P_0}(1).\label{lemma1_3}
\end{eqnarray}
Substituting \eqref{lemma1_2} in \eqref{lemma1_3}, we get the desired result.
\end{proof}
Now our objective is to prove \eqref{cond2}. We define the measure $Q_{\bm\gamma}(A)=P_0\left({p_{\bm\gamma}}/{p_{\bm\gamma_0}}\mathrm{l}\!\!\!1_A\right)$ and the corresponding density $q_{\bm\gamma}=p_0p_{\bm\gamma}/p_{\bm\gamma_0}$ as given in \citet{kleijn2012bernstein}. Also, we define a measure $Q_{\bm\gamma,n}$ by $Q_{\bm\gamma,n}(A)=P_0\left({p_{\bm\gamma,n}}/{p_{\bm\gamma_0,n}}\mathrm{l}\!\!\!1_A\right)$ with $q_{\bm\gamma,n}=p_0p_{\bm\gamma,n}/p_{\bm\gamma_0,n}$. The misspecified Kullback-Leibler neighborhood of $\bm\gamma_0$ is defined as
\begin{eqnarray}
B(\epsilon,\bm\gamma_0,P_0)=\left\{\bm\gamma\in\bm\Theta\times U:-P_0\log\left(\frac{p_{\bm\gamma,n}}{p_{\bm\gamma_0,n}}\right)\leq\epsilon^2,P_0\left(\log\left(\frac{p_{\bm\gamma,n}}{p_{\bm\gamma_0,n}}\right)\right)^2\leq\epsilon^2\right\}\nonumber
\end{eqnarray}
By Theorem 3.1 of \citet{kleijn2012bernstein}, condition \eqref{cond2} is satisfied if we can ensure that for every $\epsilon>0$, there exists a sequence of tests $\{\phi_n\}$ such that
\begin{eqnarray}
P^{(n)}_0\phi_n\rightarrow0,\,\,\,\sup_{\{\bm\gamma:\|\bm\gamma-\bm\gamma_0\|\geq\epsilon\}}Q^{(n)}_{\bm\gamma,n}(1-\phi_n)\rightarrow0.\label{cond3}
\end{eqnarray}
The above condition is ensured by the next Lemma.

%%%%%%%%%%Proof of Lemma 3

\begin{lemma}
Assume that $\bm\gamma_0$ is a unique point of minimum of $\bm\gamma\mapsto-P_0\log p_{\bm\gamma}$. Then there exist tests $\phi_n$ satisfying \eqref{cond3}.
\end{lemma}
\begin{proof}
For given $\bm\gamma_1\neq\bm\gamma_0$ consider the tests $\phi_{n,\bm\gamma_1}=\mathrm{l}\!\!\!1\{\mathbb{P}_n\log(p_0/q_{\bm\gamma_1,n})<0\}$.
Note that $\mathbb{P}_n\log(p_0/q_{\bm\gamma_1,n})=\mathbb{P}_n\log(p_0/q_{\bm\gamma_1})+O_{P_0}(r^{-4}_n)\stackrel{P^{(n)}_0}\rightarrow P_0\log(p_0/q_{\bm\gamma_1})$
and $P_0\log(p_0/q_{\bm\gamma_1})=P_0\log(p_{\bm\gamma_0}/p_{\bm\gamma_1})>0$ for $\bm\gamma_1\neq\bm\gamma_0$ by the definition of $\bm\gamma_0$. Hence, $P^{(n)}_0\phi_{n,\bm\gamma_1}\rightarrow0$ as $n\rightarrow\infty$. By Markov's inequality we have that
\begin{eqnarray}
Q^{(n)}_{\bm\gamma,n}(1-\phi_{n,\bm\gamma_1})&=&Q^{(n)}_{\bm\gamma,n}(\exp\{sn\mathbb{P}_n\log(p_0/q_{\bm\gamma_1,n})\}>1)\nonumber\\
&\leq&Q^{(n)}_{\bm\gamma,n}\exp\{sn\mathbb{P}_n\log(p_0/q_{\bm\gamma_1,n})\}\nonumber\\
&=&\left(Q_{\bm\gamma,n}(p_0/q_{\bm\gamma_1,n})^s\right)^n=\left(\rho(\bm\gamma_1,\bm\gamma,s)+O(r^{-4}_n)\right)^n,\nonumber
\end{eqnarray}
for $\rho(\bm\gamma_1,\bm\gamma,s)=\int p^s_0q^{-s}_{\bm\gamma_1}q_{\bm\gamma}d\mu$. By \citet{kleijn2006misspecification} the function $s\mapsto\rho(\bm\gamma_1,\bm\gamma_1,s)$ converges to $P_0(q_{\bm\gamma_1}>0)=P_0(p_{\bm\gamma_1}>0)$ as $s\uparrow1$ and has left derivative $P_0\log\left(\frac{q_{\bm\gamma_1}}{p_0}\right)1\!\mathrm{l}\{q_{\bm\gamma_1}>0\}=P_0\log\left(\frac{p_{\bm\gamma_1}}{p_{\bm\gamma_0}}\right)1\!\mathrm{l}\{p_{\bm\gamma_1}>0\}$ at $s=1$. Then either $P_0(p_{\bm\gamma_1}>0)<1$ or $P_0(p_{\bm\gamma_1}>0)=1$ and $P_0\log\left(\frac{p_{\bm\gamma_1}}{p_{\bm\gamma_0}}\right)1\!\mathrm{l}\{p_{\bm\gamma_1}>0\}=P_0\log\left(\frac{p_{\bm\gamma_1}}{p_{\bm\gamma_0}}\right)<0$ or both. In either case it follows that there exists $s_{\bm\gamma_1}<1$ arbitrarily close to 1 such that $\rho(\bm\gamma_1,\bm\gamma_1,s_{\bm\gamma_1})<1$. It is easy to show that the map $\bm\gamma\mapsto\rho(\bm\gamma_1,\bm\gamma,s_{\bm\gamma_1})$
%\begin{eqnarray}
%\rho(\bm\gamma_1,\bm\gamma,s_{\bm\gamma_1})&=&P_0\left(\frac{p_{\bm\gamma}(X,Y)}{p^{s_{\bm\gamma_1}}_{\bm\gamma_1}(X,Y)p^{1-s_{\bm\gamma_1}}_{\bm\gamma_0}(X,Y)}\right)\nonumber\\
%&=&P_0\left(\exp(\sigma^{-2}(f_{\bm\theta}(t)-s_{\bm\gamma_1}f_{\bm\theta_1}(t)-(1-s_{\bm\gamma_1})f_{\bm\theta_0}(t))Y\right.\nonumber\\
%&&\left.+(2\sigma^2)^{-1}(s_{\bm\gamma_1}(f_{\bm\theta_1}(t))^2+(1-s_{\bm\gamma_1})f_{\bm\theta_1}(t)f_{\bm\theta_0}(t)-(f_{\bm\theta}(t))^2))\right),\nonumber
%\end{eqnarray}
is continuous at $\bm\gamma_1$ by the dominated convergence theorem. Therefore, for every $\bm\gamma_1$, there exists an open neighborhood $G_{\bm\gamma_1}$ such that
\begin{eqnarray}
u_{\bm\gamma_1}=\sup_{\bm\gamma\in G_{\bm\gamma_1}}\rho(\bm\gamma_1,\bm\gamma,s_{\bm\gamma_1})<1\nonumber.
\end{eqnarray}
The set $\{\bm\gamma\in\bm\Theta\times U:\|\bm\gamma-\bm\gamma_0\|\geq\epsilon\}$ is compact and hence can be covered with finitely many sets of the type $G_{\bm\gamma_i}$ for $i=1,\ldots,k$. Let us define $\phi_n=\max_{i}\{\phi_{n,\bm\gamma_i}:\,i=1,\ldots,k\}$.
This test satisfies $P^{(n)}_0\phi_n\leq\sum_{i=1}^kP^{(n)}_0\phi_{n,\bm\gamma_i}\rightarrow0$, and
\begin{eqnarray}
Q^{(n)}_{\bm\gamma,n}(1-\phi_n)&\leq&\max_{i=1,\ldots,k}Q^{(n)}_{\bm\gamma,n}(1-\phi_{n,\bm\gamma_i})\leq\max_{i=1,\ldots,k}(u_{\bm\gamma_i}+O(r^{-4}_n))^n\rightarrow0\nonumber
\end{eqnarray}
uniformly in $\bm\gamma\in\cup_{i=1}^kG_{\bm\gamma_i}$. Therefore, the tests $\phi_n$ meet \eqref{cond3}.
\end{proof}
The proof of Theorem 3.1 of \citet{kleijn2012bernstein} also uses the results of the next two lemmas.

%%%%%%%%%%Proof of Lemma 4

\begin{lemma}
Suppose that $P_0\dot{\bm\ell}_{\bm\gamma_0}\dot{\bm\ell}^T_{\bm\gamma_0}$ is invertible. Then for every sequence $\{M_n\}$ such that $M_n\rightarrow\infty$, there exists a sequence of tests $\{\omega_n\}$ such that for some constant $D>0$, $\epsilon>0$ and large enough $n$,
\begin{eqnarray}
P^{(n)}_0\omega_n\rightarrow0,\,\,\,Q^{(n)}_{\bm\gamma,n}(1-\omega_n)\leq e^{-nD(\|\bm\gamma-\bm\gamma_0\|^2\wedge\epsilon^2)},\nonumber
\end{eqnarray}
for all $\bm\gamma\in\Theta\times U$ such that $\|\bm\gamma-\bm\gamma_0\|\geq M_n/\sqrt{n}$.
\end{lemma}
\begin{proof}
Let $\{M_n\}$ be given. We construct two sequences of tests. The first sequence is used to test $P_0$ versus $\{Q_{\bm\gamma,n}:\bm\gamma\in(\Theta\times U)_1\}$ with $(\Theta\times U)_1=\{\bm\gamma\in\Theta\times U:M_n/\sqrt{n}\leq \|\bm\gamma-\bm\gamma_0\|\leq\epsilon\}$ and the second to test $P_0$ versus $\{Q_{\bm\gamma,n}:\bm\gamma\in(\Theta\times U)_2\}$ with $(\Theta\times U)_2=\{\bm\gamma\in\Theta\times U:\|\bm\gamma-\bm\gamma_0\|>\epsilon\}$. These two sequences are combined to test $P_0$ versus $\{Q_{\bm\gamma,n}:\|\bm\gamma-\bm\gamma_0\|\geq M_n/\sqrt{n}\}$.\\
To construct the first sequence, a constant $L>0$ is chosen to truncate the score-function, that is, $\dot{\bm\ell}^L_{\gamma_0}=0$ if $\|\dot{\bm\ell}_{\bm\gamma_0}\|>L$ and $\dot{\bm\ell}^L_{\bm\gamma_0}=\dot{\bm\ell}_{\bm\gamma_0}$ otherwise. Similarly we define $\dot{\bm\ell}^L_{\bm\gamma_0,n}$. We define
\begin{eqnarray}
\omega_{1,n}=\mathrm{l}\!\!\!1\left\{\|(\mathbb{P}_n-P_0)\dot{\bm\ell}^L_{\bm\gamma_0,n}\|>\sqrt{M_n/{n}}\right\}.\nonumber
\end{eqnarray}
Since the function $\dot{\bm\ell}_{\bm\gamma_0}$ is square-integrable, we observe that the matrices $P_0\dot{\bm\ell}_{\bm\gamma_0,n}\dot{\bm\ell}^T_{\bm\gamma_0,n}$, $P_0\dot{\bm\ell}_{\bm\gamma_0,n}(\dot{\bm\ell^L}_{\bm\gamma_0,n})^T$ and $P_0\dot{\bm\ell}^L_{\bm\gamma_0,n}(\dot{\bm\ell^L})^T_{\bm\gamma_0,n}$ can be made sufficiently close to each other for sufficiently large choices of $L$ and $n$. We fix such an $L$. Now,
\begin{eqnarray}
P^{(n)}_0\omega_{1,n}&=&P^{(n)}_0\left(\|\sqrt{n}(\mathbb{P}_n-P_0)\dot{\bm\ell}^L_{\bm\gamma_0,n}\|^2>M_n\right)\nonumber\\
&\leq&P^{(n)}_0\left(\|\sqrt{n}(\mathbb{P}_n-P_0)\dot{\bm\ell}^L_{\bm\gamma_0}\|^2>M_n/4\right)\nonumber\\
&&+P^{(n)}_0\left(\|\sqrt{n}(\mathbb{P}_n-P_0)(\dot{\bm\ell}^L_{\bm\gamma_0,n}-\dot{\bm\ell}^L_{\bm\gamma_0})\|^2>M_n/4\right).\nonumber
%&=&P^n_0\left(O_{P_0}(1)>M_n/2\right)+P^n_0\left(O_{P_0}(r^{-1}_n)>M_n/2\right)\rightarrow0.\nonumber
\end{eqnarray}
The right hand side of the above inequality converges to zero since both sequences inside the brackets are stochastically bounded. The rest of the proof follows from the proof of Theorem 3.3 of \citet{kleijn2012bernstein} and Lemma $2$.
As far as $Q^{(n)}_{\bm\gamma,n}(1-\omega_{1,n})$ for $\bm\gamma\in\left(\Theta\times U\right)_1$ is concerned, for all $\bm\gamma$
\begin{eqnarray}
\lefteqn{Q^{(n)}_{\bm\gamma,n}\left(\|(\mathbb{P}_n-P_0)\dot{\bm\ell}^L_{\bm\gamma_0,n}\|\leq\sqrt{M_n/{n}}\right)}\nonumber\\
&=&Q^{(n)}_{\bm\gamma,n}\left(\sup_{\bm v\in S}\bm v^T(\mathbb{P}_n-P_0)\dot{\bm\ell}^L_{\bm\gamma_0,n}\leq\sqrt{M_n/{n}}\right)\nonumber\\
&\leq&\inf_{\bm v\in S}Q^{(n)}_{\bm\gamma,n}\left(\bm v^T(\mathbb{P}_n-P_0)\dot{\bm\ell}^L_{\bm\gamma_0,n}\leq\sqrt{M_n/{n}}\right),\nonumber
\end{eqnarray}
where $S$ is the unit sphere in $\mathbb{R}^{p+1}$. Choosing $\bm v=(\bm\gamma-\bm\gamma_0)/\|\bm\gamma-\bm\gamma_0\|$, the right hand side of the previous display can be bounded by
\begin{eqnarray}
\lefteqn{Q^{(n)}_{\bm\gamma,n}\left((\bm\gamma-\bm\gamma_0)^T(\mathbb{P}_n-P_0)\dot{\bm\ell}^L_{\bm\gamma_0,n}\leq\sqrt{M_n/n}\|\bm\gamma-\bm\gamma_0\|\right)}\nonumber\\
&=&Q^{(n)}_{\bm\gamma,n}\left((\bm\gamma_0-\bm\gamma)^T(\mathbb{P}_n-\tilde{Q}_{\bm\gamma,n})\dot{\bm\ell}^L_{\bm\gamma_0,n}\geq (\bm\gamma-\bm\gamma_0)^T(\tilde{Q}_{\bm\gamma,n}-\tilde{Q}_{\bm\gamma_0,n})\dot{\bm\ell}^L_{\bm\gamma_0,n}\right.\nonumber\\
&&\left.-\sqrt{M_n/n}\|\bm\gamma-\bm\gamma_0\|\right),\nonumber
\end{eqnarray}
where $\tilde{Q}_{\bm\gamma,n}=\|{Q}_{\bm\gamma,n}\|^{-1}Q_{\bm\gamma,n}$ and also note that $P_0=Q_{\bm\gamma_0,n}=\tilde{Q}_{\bm\gamma_0,n}$. It should be noted that
\begin{eqnarray}
\lefteqn{(\bm\gamma-\bm\gamma_0)^T(\tilde{Q}_{\bm\gamma,n}-\tilde{Q}_{\bm\gamma_0,n})\dot{\bm\ell}^L_{\bm\gamma_0,n}}\nonumber\\
&=&\left(P_0(p_{\bm\gamma,n}/p_{\bm\gamma_0,n})\right)^{-1}(\bm\gamma-\bm\gamma_0)^T\left(P_0\left((p_{\bm\gamma,n}/p_{\bm\gamma_0,n}-1)\dot{\bm\ell}^L_{\bm\gamma_0,n}\right)\right.\nonumber\\
&&\left.+\left(1-P_0(p_{\gamma,n}/p_{\bm\gamma_0,n})\right)P_0\dot{\bm\ell}^L_{\bm\gamma_0,n}\right).\nonumber
\end{eqnarray}
By Lemma 3.4 of \citet{kleijn2012bernstein}, $$\left(P_0(p_{\bm\gamma,n}/p_{\bm\gamma_0,n}-1)\right)= O\left(\|\bm\gamma-\bm\gamma_0\|^2\right)$$ as $\bm\gamma\rightarrow\bm\gamma_0$. Using the differentiability of $\bm\gamma\mapsto\log(p_{\bm\gamma,n}/p_{\bm\gamma_0,n})$ and Lemma 3.4 of \citet{kleijn2012bernstein}, we see that
\begin{eqnarray}
\lefteqn{P_0\left\|\left(\frac{p_{\bm\gamma,n}}{p_{\bm\gamma_0,n}}-1-(\bm\gamma-\bm\gamma_0)^T\dot{\bm\ell}_{\bm\gamma_0,n}\right)\dot{\bm\ell}^L_{\bm\gamma_0,n}\right\|}\nonumber\\
&\leq&P_0\left\|\left(\frac{p_{\bm\gamma,n}}{p_{\bm\gamma_0,n}}-1-\log\frac{p_{\bm\gamma,n}}{p_{\bm\gamma_0,n}}\right)\dot{\bm\ell}^L_{\bm\gamma_0,n}\right\|\nonumber\\
&&+P_0\left\|\left(\log\frac{p_{\gamma,n}}{p_{\gamma_0,n}}-(\bm\gamma-\bm\gamma_0)^T\dot{\bm\ell}_{\bm\gamma_0,n}\right)\dot{\bm\ell}^L_{\bm\gamma_0,n}\right\|,\label{inter_lemma4}
\end{eqnarray}
which is $o(\|\bm\gamma-\bm\gamma_0\|)$. Also note that for all $\bm\gamma\in\left(\Theta\times U\right)_1$, $$-\|\bm\gamma-\bm\gamma_0\|\sqrt{M_n/n}\geq -\|\bm\gamma-\bm\gamma_0\|^2(M_n)^{-1/2}.$$ Then we observe that for every $\delta>0$, there exist $\epsilon>0$, $L>0$ and $N\geq 1$ such that for all $n\geq N$ and all $\bm\gamma\in\left(\Theta\times U\right)_1$,
\begin{eqnarray}
&&(\bm\gamma-\bm\gamma_0)^T(\tilde{Q}_{\bm\gamma,n}-\tilde{Q}_{\bm\gamma_0,n})\dot{\bm\ell}^L_{\bm\gamma_0,n}-\sqrt{M_n/n}\|\bm\gamma-\bm\gamma_0\|\nonumber\\
&&\geq(\bm\gamma-\bm\gamma_0)^TP_0(\dot{\bm\ell}_{\gamma_0,n}\dot{\bm\ell}^T_{\bm\gamma_0,n})(\bm\gamma-\bm\gamma_0)-\delta\|\bm\gamma-\bm\gamma_0\|^2.\nonumber
\end{eqnarray}
Denoting $\Delta(\bm\gamma)=(\bm\gamma-\bm\gamma_0)^TP_0(\dot{\bm\ell}_{\gamma_0,n}\dot{\bm\ell}^T_{\bm\gamma_0,n})(\bm\gamma-\bm\gamma_0)$ and using the positive definiteness of $P_0(\dot{\bm\ell}_{\gamma_0,n}\dot{\bm\ell}^T_{\bm\gamma_0,n})$ for sufficiently large $n$, there exists a positive constant $c$ such that $-\delta\|\bm\gamma-\bm\gamma_0\|^2\geq-\delta/c\Delta(\bm\gamma)$. Also, there exists a constant $r(\delta)$ which depends only on $P_0(\dot{\bm\ell}_{\gamma_0,n}\dot{\bm\ell}^T_{\bm\gamma_0,n})$ and has the property that $r(\delta)\rightarrow1$ if $\delta\rightarrow0$. We can choose such an $r(\delta)$ to satisfy
\begin{eqnarray}
Q^{(n)}_{\bm\gamma,n}(1-\omega_{1,n})\leq Q^{(n)}_{\bm\gamma,n}\left((\bm\gamma_0-\bm\gamma)^T(\mathbb{P}_n-\tilde{Q}_{\bm\gamma,n})\dot{\bm\ell}^L_{\bm\gamma_0,n}\geq r(\delta)\Delta(\bm\gamma)\right),\nonumber
\end{eqnarray}
for sufficiently small $\epsilon$, sufficiently large $L$ and $n$, making the type-II error bounded above by the unnormalized tail probability $Q^{(n)}_{\bm\gamma,n}(\bar{W}_n\geq r(\delta)\Delta(\bm\gamma))$ where $W_i=(\bm\gamma-\bm\gamma_0)^T\left(\dot{\bm\ell}^L_{\bm\gamma_0,n}(X_i,Y_i)-\tilde{Q}_{\bm\gamma,n}\dot{\bm\ell}^L_{\bm\gamma_0,n}\right)$, $(1\leq i\leq n)$. We note that $\tilde{Q}_{\bm\gamma,n}W_i=0$ and $W_i$ are independent. Also,
\begin{eqnarray}
|W_i|\leq \|\bm\gamma-\bm\gamma_0\|\left(\|\dot{\bm\ell}^L_{\bm\gamma_0,n}(X_i,\bm Y_i)|+\|\tilde{Q}_{\bm\gamma,n}\dot{\bm\ell}^L_{\bm\gamma_0,n}\|\right)\leq 2L\sqrt{p+1}\|\bm\gamma-\bm\gamma_0\|.\nonumber
\end{eqnarray}
Then we have
\begin{eqnarray}
\lefteqn{\Var_{\tilde{Q}_{\bm\gamma,n}}W_i}\nonumber\\
&=&(\bm\gamma-\bm\gamma_0)^T\left[\tilde{Q}_{\bm\gamma,n}\left(\dot{\bm\ell}^L_{\bm\gamma_0,n}(\dot\ell^L_{\bm\gamma_0,n})^T\right)-\tilde{Q}_{\bm\gamma,n}\dot\ell^L_{\bm\gamma_0,n}\tilde{Q}_{\bm\gamma,n}(\dot\ell^L_{\bm\gamma_0,n})^T\right](\bm\gamma-\bm\gamma_0)\nonumber\\
&\leq&(\bm\gamma-\bm\gamma_0)^T\tilde{Q}_{\bm\gamma,n}\left(\dot\ell^L_{\bm\gamma_0,n}(\dot\ell^L_{\bm\gamma_0,n})^T\right)(\bm\gamma-\bm\gamma_0)\nonumber\\
&=&\left(P_0(p_{\bm\gamma,n}/p_{\bm\gamma_0,n})\right)^{-1}(\bm\gamma-\bm\gamma_0)^TP_0\left(\left(p_{\bm\gamma,n}/p_{\bm\gamma_0,n}-1\right)\dot{\bm\ell}^L_{\bm\gamma_0,n}(\dot{\bm\ell}^L_{\bm\gamma_0,n})^T\right)(\bm\gamma-\bm\gamma_0)\nonumber\\
&&+\left(P_0(p_{\bm\gamma,n}/p_{\bm\gamma_0,n})\right)^{-1}(\bm\gamma-\bm\gamma_0)^TP_0\left(\dot{\bm\ell}^L_{\bm\gamma_0,n}(\dot{\bm\ell}^L_{\bm\gamma_0,n})^T\right)(\bm\gamma-\bm\gamma_0).\nonumber
\end{eqnarray}
The first term on the right side above is $o\left(\|\bm\gamma-\bm\gamma_0\|^2\right)$ by similar argument as in \eqref{inter_lemma4}. Then it follows that $\Var_{\tilde{Q}_{\bm\gamma,n}}W_i\leq s(\delta)\Delta(\bm\gamma)$ for small enough $\epsilon$ and large enough $L$, where $s(\delta)\rightarrow1$ as $\delta\rightarrow0$ for $i=1,\ldots,n$. We apply Bernstein's inequality to obtain
\begin{eqnarray}
Q^{(n)}_{\bm\gamma,n}(1-\omega_{1,n})&=&\|Q_{\bm\gamma,n}\|^n\tilde{Q}^{(n)}_{\bm\gamma,n}\left(W_1+\cdots+W_n\geq nr(\delta)\Delta(\bm\gamma)\right)\nonumber\\
&\leq&\|Q_{\bm\gamma,n}\|^n\exp\left(-\frac{1}{2}\frac{r^2(\delta)n\Delta(\bm\gamma)}{s(\delta)+1.5L\sqrt{p+1}\|\bm\gamma-\bm\gamma_0\|r(\delta)}\right).\nonumber
\end{eqnarray}
We can make the factor $t(\delta)=r^2(\delta)\left(s(\delta)+1.5L\sqrt{p+1}\|\bm\gamma-\bm\gamma_0\|r(\delta)\right)^{-1}$ arbitrarily close to $1$ for sufficiently small $\delta$ and $\epsilon$. By Lemma 3.4 of \citet{kleijn2012bernstein}, we have
\begin{eqnarray}
\|Q_{\bm\gamma,n}\|&=&1+P_0\log\frac{p_{\bm\gamma,n}}{p_{\bm\gamma_0,n}}+\frac{1}{2}P_0\left(\log\frac{p_{\bm\gamma,n}}{p_{\bm\gamma_0,n}}\right)^2+o\left(\|\bm\gamma-\bm\gamma_0\|^2\right)\nonumber\\
&\leq&1+P_0\log\frac{p_{\bm\gamma,n}}{p_{\bm\gamma_0,n}}+\frac{1}{2}(\bm\gamma-\bm\gamma_0)^TP_0\left(\dot{\bm\ell}_{\bm\gamma_0,n}\dot{\bm\ell}^T_{\bm\gamma_0,n}\right)(\bm\gamma-\bm\gamma_0)+o\left(\|\bm\gamma-\bm\gamma_0\|^2\right)\nonumber\\
&\leq&1-\frac{1}{2}(\bm\gamma-\bm\gamma_0)^T\bm V_{\bm\gamma_0}(\bm\gamma-\bm\gamma_0)+\frac{1}{2}u(\delta)\Delta(\bm\gamma),\nonumber
\end{eqnarray}
for some constant $u(\delta)$ such that $u(\delta)\rightarrow1$ as $\delta\rightarrow0$ for large $n$. Using the inequality $1+x\leq e^x$ for all $x\in\mathbb{R}$, we have, for sufficiently small $\|\bm\gamma-\bm\gamma_0\|$,
\begin{eqnarray}
Q^{(n)}_{\bm\gamma,n}(1-\omega_{1,n})\leq\exp\left(-\frac{n}{2}(\bm\gamma-\bm\gamma_0)^T\bm V_{\bm\gamma_0}(\bm\gamma-\bm\gamma_0)+\frac{n}{2}\left(u(\delta)-t(\delta)\right)\Delta(\bm\gamma)\right).\nonumber
\end{eqnarray}
Clearly, $u(\delta)-t(\delta)\rightarrow0$ as $\delta\rightarrow0$ and $\Delta(\bm\gamma)$ is bounded above by a multiple of $\|\bm\gamma-\bm\gamma_0\|^2$. Utilizing the positive definiteness of $\bm V_{\bm\gamma_0}$, we conclude that there exists a constant $C>0$ such that for sufficiently large $L$ and $n$ and sufficiently small $\epsilon>0$,
\begin{eqnarray}
Q^{(n)}_{\bm\gamma,n}(1-\omega_{1,n})\leq\exp\left(-Cn\|\bm\gamma-\bm\gamma_0\|^2\right).\nonumber
\end{eqnarray}
By the assumption of the theorem, there exists a consistent sequence of tests for $P_0$ versus $Q_{\bm\gamma,n}$ for $\|\bm\gamma-\bm\gamma_0\|>\epsilon$. Now by Lemma 3.3 of \citet{kleijn2012bernstein}, there exists a sequence of tests $\{\omega_{2,n}\}$such that
\begin{eqnarray}
P^{(n)}_0(\omega_{2,n})\leq\exp{(-nC_1)},\,\,\,\sup_{\|\bm\gamma-\bm\gamma_0\|\geq\epsilon}Q^{(n)}_{\bm\gamma,n}(1-\omega_{2,n})\leq\exp{(-nC_2)}.\nonumber
\end{eqnarray}
We define a sequence $\{\psi_n\}$ as $\psi_n=\omega_{1,n}\vee\omega_{2,n}$ for all $n\geq1$, in which case $P^{(n)}_0\psi_n\leq P^{(n)}_0\omega_{1,n}+P^{(n)}_0\omega_{2,n}\rightarrow0$ and
\begin{eqnarray}
\sup_{\bm\gamma\in\Theta\times U}Q^{(n)}_{\bm\gamma,n}(1-\psi_n)&=&\sup_{\bm\gamma\in\left(\Theta\times U\right)_1}Q^{(n)}_{\bm\gamma,n}(1-\psi_n)\vee\sup_{\bm\gamma\in\left(\Theta\times U\right)_2}Q^{(n)}_{\bm\gamma,n}(1-\psi_n)\nonumber\\
&\leq&\sup_{\bm\gamma\in\left(\Theta\times U\right)_1}Q^{(n)}_{\bm\gamma,n}(1-\omega_{1,n})\vee\sup_{\bm\gamma\in\left(\Theta\times U\right)_2}Q^{(n)}_{\bm\gamma,n}(1-\omega_{2,n})\nonumber
\end{eqnarray}
Combining the previous bounds, we get the desired result for a suitable choice of $D>0$.
\end{proof}

%%%%%%%%%%Proof of Lemma 5

\begin{lemma}
There exists a constant $K>0$ such that the prior mass of the Kullback-Leibler neighborhoods $B(\epsilon_n, \bm\gamma_0, P_0)$ satisfies $\Pi\left(B(\epsilon_n, \bm\gamma_0, P_0)\right)\geq K\epsilon^p_n$,
where $\epsilon_n\gg n^{-1/2}$.
\end{lemma}
\begin{proof}
From the proof of Lemma $2$ we get $$-P_0\log(p_{\bm\gamma,n}/p_{\bm\gamma_0,n})=O(\|\bm\gamma-\bm\gamma_0\|^2)+O(\|\bm\gamma-\bm\gamma_0\|r^{-4}_n)\leq c_1\|\bm\gamma-\bm\gamma_0\|^2+c_2\|\bm\gamma-\bm\gamma_0\|\epsilon_n$$ for sufficiently large $n$ and positive constants $c_1$ and $c_2$. Again, $P_0\left(\log(p_{\bm\gamma,n}/p_{\bm\gamma_0,n})\right)^2\leq c_3\|\bm\gamma-\bm\gamma_0\|^2$ for some constant $c_3>0$. Let $c=\min\left((2c_1)^{-1/2},(2c_2)^{-1},c^{-1/2}_3\right)$. Then $\{\bm\gamma\in\Theta\times U:\|\bm\gamma-\bm\gamma_0\|\leq c\epsilon_n\}\subset B(\epsilon_n, \bm\gamma_0, P_0)$. Since the Lebesgue-density $\pi$ of the prior is continuous and strictly positive in $\bm\gamma_0$, we see that there exists a $\delta'>0$ such that for all $0<\delta\leq\delta'$, $\Pi\left(\bm\gamma\in\Theta\times U:\|\bm\gamma-\bm\gamma_0\|\leq\delta\right)\geq\frac{1}{2}V\pi(\bm\gamma_0)\delta^{p+1}>0$, $V$ being the Lebesgue-volume of the $(p+1)$-dimensional ball of unit radius. Hence, for sufficiently large $n$, $c\epsilon_n\leq\delta'$ and we obtain the desired result.
\end{proof}
The next lemma is used to estimate the bias of the Bayes estimator in RKTB.

%%%%%%%%%%Proof of Lemma 6

\begin{lemma}
For $m\geq3$ and $n^{1/(2m)}\ll k_n\ll n^{1/2}$, $$\sup_{t\in[0,1]}|\E(f(t)|\bm X,\bm Y,\sigma^2)-f_{0}(t)|^2=O_{P_0}(k_n^2/n)+O_{P_0}(k^{1-2m}_n).$$
\end{lemma}
\begin{proof}
By \eqref{posterior},
\begin{equation}
\E(f(t)|\bm X,\bm Y,\sigma^2)=(\bm N(t))^T{(\bm X^T_n\bm X_n+k_nn^{-2}\bm I_{k_n+m-1})}^{-1}\bm X^T_n\bm Y.\label{postexp}
\end{equation}
By Lemma 12 in the appendix, we have uniformly over $t\in[0,1]$,
\begin{eqnarray}
(\bm N(t))^T{({\bm X^T_n\bm X_n})}^{-1}\bm N(t)
\asymp\frac{k_n}{n}(1+o_{P_0}(1)).\label{var}
\end{eqnarray}
Since $ f_0\in C^{m}$, there exists a $\bm\beta^*$ \citep[Theorem XII.4, page 178]{de1978practical} such that
\begin{equation}
\sup_{t\in[0,1]}|f_{0}(t)-(\bm N(t))^T\bm\beta^*|=O({k_n^{-m}}).\label{spldis}
\end{equation}
We can bound $\sup_{t\in[0,1]}|\E(f(t)|\bm X,\bm Y,\sigma^2)-f_{0}(t)|^2$ up to a constant multiple by
\begin{eqnarray}
\lefteqn{\sup_{t\in[0,1]}\left|\E(f(t)|\bm X,\bm Y,\sigma^2)-\bm (N(t))^T(\bm X^T_n\bm X_n)^{-1}\bm X^T_n\bm Y\right|^2}\nonumber\\
&&+\sup_{t\in[0,1]}\left|(\bm N(t))^T{({\bm X^T_n\bm X_n})}^{-1}\bm X^T_n(\bm Y-f_0(\bm x))\right|^2\nonumber\\
%&&+\sup_{t\in[0,1]}\left|{\left(1+\frac{k_n\sigma^2}{n}\right)}^{-1}(\bm N(t))^T\bm\beta^*-(\bm N(t))^T\bm\beta^*\right|^2\nonumber\\
&&+\sup_{t\in[0,1]}|(\bm N(t))^T{({\bm X^T_n\bm X_n})}^{-1}\bm X^T_n(f_{0}(\bm{x})-\bm X_n\bm\beta^*)|^2\nonumber\\
&&+\sup_{t\in[0,1]}|f_{0}(t)-(\bm N(t))^T\bm\beta^*|^2.\label{bias_1}
\end{eqnarray}
%\end{small}
Using the Binomial Inverse Theorem, the Cauchy-Schwarz inequality and \eqref{var}, the first term of \eqref{bias_1} can be shown to be $O_{P_0}(k^6_n/n^{8})$. The second term can be bounded up to a constant multiple by
\begin{eqnarray}
\lefteqn{\max_{1\leq k\leq n}\left|(\bm N(s_k))^T{({\bm X^T_n\bm X_n})}^{-1}\bm X^T_n\bm\varepsilon\right|^2}\nonumber\\
&&+\sup_{t,t':|t-t'|\leq n^{-1}}\left|(\bm N(t)-\bm N(t'))^T{({\bm X^T_n\bm X_n})}^{-1}\bm X^T_n\bm\varepsilon\right|^2\label{bias_2},
\end{eqnarray}
where $s_k=k/n$ for $k=1,\ldots,n$. Applying the mean value theorem to the second term of the above sum, we can bound the expression by a constant multiple of
\begin{eqnarray}
\max_{1\leq k\leq n}\left|(\bm N(s_k))^T{({\bm X^T_n\bm X_n})}^{-1}\bm X^T_n\bm\varepsilon\right|^2
+\sup_{t\in[0,1]}\frac{1}{n^2}\left|\left(\bm N^{(1)}(t)\right)^T{({\bm X^T_n\bm X_n})}^{-1}\bm X^T_n\bm\varepsilon\right|^2.\nonumber
\end{eqnarray}
By the spectral decomposition, we can write $\bm X_n{({\bm X^T_n\bm X_n})}^{-1}\bm X^T_n=\bm P^T\bm D\bm P$, where $\bm P$ is an orthogonal matrix and $\bm D$ is a diagonal matrix with $k_n+m-1$ ones and $n-k_n-m+1$ zeros in the diagonal. Now using the Cauchy-Schwarz inequality, we get
\begin{eqnarray}
\lefteqn{\max_{1\leq k\leq n}\left|(\bm N(s_k))^T{({\bm X^T_n\bm X_n})}^{-1}\bm X^T_n\bm\varepsilon\right|^2}\nonumber\\
&\leq&\max_{1\leq k\leq n}\left\{(\bm N(s_k))^T{({\bm X^T_n\bm X_n})}^{-1}\bm N(s_k)\right\}
\bm\varepsilon^T\bm P^T\bm D\bm P\bm\varepsilon.\nonumber
\end{eqnarray}
Note that $\Var_0(\bm P\bm\varepsilon)=\E_0(\Var(\bm P\bm\varepsilon|
\bm X))+\Var_0(\E(\bm P\bm\varepsilon|
\bm X))=\sigma^2_0\bm I_{k_n+m-1}$. Hence, we get $\E_0(\bm\varepsilon^T\bm P^T\bm D\bm P\bm\varepsilon)=\sigma^2_0(k_n+m-1)$. In view of Lemma 12, we can conclude that the first term of \eqref{bias_2} is $O_{P_0}(k_n^{2}/n)$. Again applying the Cauchy-Schwarz inequality, the second term of \eqref{bias_2} is bounded by
\begin{eqnarray}
\sup_{t\in[0,1]}\left\{\frac{1}{n^2}\left(\bm N^{(1)}(t)\right)^T{(\bm X_n^T\bm X_n)}^{-1} \bm N^{(1)}(t)\right\}(\bm\varepsilon^T\bm\varepsilon),\nonumber
\end{eqnarray}
which is $O_{P_0}\left(n(k_n^{3}/n^3)\right)=O_{P_0}\left(k_n^{3}/n^2\right)$, using Lemma 12.
Thus, the second term of \eqref{bias_1} is $O_{P_0}\left(k_n^{2}/n\right)$.
 Using the Cauchy-Schwarz inequality, \eqref{var} and \eqref{spldis}, the third term of \eqref{bias_1} is $O_{P_0}\left(k^{1-2m}_n\right)$. The fourth term of \eqref{bias_1} is of the order of $k^{-2m}_n$ as a result of \eqref{spldis}.
\end{proof}
The following lemma controls posterior variability in RKTB.

%%%%%%%%%%Proof of Lemma 7

\begin{lemma}
If $m\geq3$ and $n^{1/(2m)}\ll k_n\ll n^{1/2}$, then for all $\epsilon>0$,
\begin{eqnarray}
\Pi^*_n\left(\sup_{t\in[0,1]}|f(t)-f_{0}(t)|>\epsilon|\bm X,\bm Y,\sigma^2\right)=o_{P_0}(1).\nonumber
\end{eqnarray}
\end{lemma}
\begin{proof}
By Markov's inequality and the fact that $|a+b|^2\leq 2(|a|^2+|b|^2)$ for two real numbers $a$ and $b$, we can bound
$\Pi^*_n\left(\sup_{t\in[0,1]}|f(t)-f_{0}(t)|>\epsilon|\bm X,\bm Y,\sigma^2\right)$ by
\begin{eqnarray}
\lefteqn{2\epsilon^{-2}\left\{\sup_{t\in[0,1]}\left|\E(f(t)|\bm X,\bm Y,\sigma^2)-f_{0}(t)\right|^2\right.}\nonumber\\
&&\left.+\E\left[\sup_{t\in[0,1]}\left|f(t)-\E(f(t)|\bm X,\bm Y,\sigma^2)\right|^2\large{|}\bm X,\bm Y,\sigma^2\right]\right\}.\label{chv}
\end{eqnarray}
By Lemma 6, the first term inside the bracket above is $O_{P_0}(k_n^2/n)+O_{P_0}(k^{1-2m}_n)$. For $\bm \varepsilon^*:={(\bm X^T_n\bm X_n+k_nn^{-2}\bm I_{k_n+m-1})}^{1/2}\bm\beta-{(\bm X^T_n\bm X_n+k_nn^{-2}\bm I_{k_n+m-1})}^{-1/2}\bm X^T_n\bm Y$, we have $\bm\varepsilon^*|\bm X,\bm Y,\sigma^2\sim N(\bm 0, \sigma^2\bm I_{k_n+m-1})$. Writing $$\sup_{t\in[0,1]}|f(t)-\E[f(t)|\bm X,\bm Y,\sigma^2]|=
\sup_{t\in[0,1]}\left|(\bm N(t))^T{(\bm X^T_n\bm X_n+k_nn^{-2}\bm I_{k_n+m-1})}^{-1/2}\bm\varepsilon^*\right|$$
and using the Cauchy-Schwarz inequality and Lemma 12, the second term inside the bracket in \eqref{chv} is seen to be $O_{P_0}(k_n^{2}/n)$. By the assumed conditions on $m$ and $k_n$, the lemma follows.
\end{proof}
The next lemma proves the posterior consistency of $\bm\theta$ using the results of Lemmas $6$ and $7$.

%%%%%%%%%%Proof of Lemma 8

\begin{lemma}
If $m\geq3$ and $n^{1/(2m)}\ll k_n\ll n^{1/2}$, then for all $\epsilon>0$, $\Pi^*_n(\|\bm\theta-\bm\theta_0\|>\epsilon|\bm X,\bm Y,\sigma^2)=o_{P_0}(1).$
\end{lemma}
\begin{proof}
By the triangle inequality,
\begin{eqnarray}
|R_{f,n}(\bm\eta)-R_{f_0}(\bm\eta)|&\leq&\|f(\cdot)-f_0(\cdot)\|_g+\|f_{\bm\eta,r_n}(\cdot)-f_{\bm\eta}(\cdot)\|_g\nonumber\\
&\leq&c'_1\sup_{t\in[0,1]}|f(t)-f_{0}(t)|+c'_2r^{-4}_n,\nonumber
\end{eqnarray}
for appropriately chosen constants $c'_1$ and $c'_2$. For a sequence $\tau_n\rightarrow0$, define $$T_n=\{f:\sup_{t\in[0,1]}|f(t)-f_0(t)|\leq\tau_n\}.$$ By Lemma $7$ we can choose $\tau_n$ to satisfy $\Pi(T^c_n|\bm X,\bm Y,\sigma^2)=o_{P_0}(1)$. Hence for $f\in T_n$,
\begin{eqnarray}
\sup_{\bm\eta\in\Theta}|R_{f,n}(\bm\eta)-R_{f_0}(\bm\eta)|\leq c'_1\tau_n+c'_2r^{-4}_n=o(1).\nonumber
\end{eqnarray}
Therefore, for any $\delta>0$, $\Pi^*_n\left(\sup_{\bm\eta\in\Theta}|R_{f,n}(\bm\eta)-R_{f_0}(\bm\eta)|>\delta|\bm X,
\bm Y,\sigma^2\right)=o_{P_0}(1)$. By assumption \eqref{assmp}, for $\|\bm\theta-\bm\theta_0\|\geq\epsilon$ there exists a $\delta>0$ such that
\begin{eqnarray}
\delta<R_{f_0}(\bm\theta)-R_{f_0}(\bm\theta_0)
&\leq&R_{f_0}(\bm\theta)-R_{f,n}(\bm\theta)+R_{f,n}(\bm\theta_0)-R_{f_0}(\bm\theta_0)\nonumber\\
&\leq&2\sup_{\bm\eta\in\Theta}|R_{f,n}(\bm\eta)-R_{f_0}(\bm\eta)|,\nonumber
\end{eqnarray}
since $R_{f,n}(\bm\theta)\leq R_{f,n}(\bm\theta_0)$. Consequently,
\begin{eqnarray}
\Pi^*_n(\|\bm\theta-\bm\theta_0\|>\epsilon|\bm X,\bm Y,\sigma^2)&\leq&\Pi^*_n\left(\sup_{\bm\eta\in\Theta}|R_{f,n}(\bm\eta)-R_{f_0}(\bm\eta)|>{\delta}/{2}|\bm X,\bm Y,\sigma^2\right)\nonumber\\
&=&o_{P_0}(1).\nonumber
\end{eqnarray}
\end{proof}
In the following lemma we approximate $\sqrt{n}(\bm\theta-\bm\theta_0)$ by a linear functional of $f$ which is later used in Theorem $4.2$ to obtain the limiting posterior distribution of $\sqrt{n}(\bm\theta-\bm\theta_0)$.

%%%%%%%%%%Proof of Lemma 9

\begin{lemma}
Let $m$ be an integer greater than or equal to $3$ and $n^{1/(2m)}\ll k_n\ll n^{1/2}$. Then there exists $E_n\subseteq C^m((0,1))\times\bm\Theta$ with $\Pi(E^c_n|\bm X,\bm Y,\sigma^2)=o_{P_0}(1)$, such that uniformly for $(f, \bm\theta)\in E_n$,
\begin{eqnarray}
\|\sqrt{n}(\bm\theta-\bm\theta_0)-\bm{J}_{\bm\theta_0}^{-1}\sqrt{n}(\bm\Gamma(f)-\bm\Gamma(f_{0}))\|&\lesssim&\sqrt{n}r^{-4}_n,\label{thm1}
\end{eqnarray}
where $\bm\Gamma(z)=\int_0^1 \left(\dot f_{\bm\theta_0}(t)\right)^Tz(t)g(t)dt$.
\end{lemma}
\begin{proof}
By definitions of $\bm\theta$ and $\bm\theta_0$,
\begin{eqnarray}
\int_0^1\left(\dot f_{\theta,r_n}(t)\right)^T(f(t)-f_{\theta,r_n}(t))g(t)dt&=&\bm 0,\label{theta_1}\\
\int_0^1\left(\dot f_{\theta_0}(t)\right)^T(f_0(t)-f_{\theta_0}(t))g(t)dt&=&\bm 0.\label{theta0}
\end{eqnarray}
We can rewrite \eqref{theta_1} as
\begin{eqnarray}
\lefteqn{\int_0^1\left(\dot f_{\theta_0}(t)\right)^T(f(t)-f_{\bm\theta}(t))g(t)dt
+\int_0^1\left(\dot f_{\theta}(t)-\dot f_{\theta_0}(t)\right)^T(f(t)- f_{\bm\theta}(t))g(t)dt}\nonumber\\
&&+\int_0^1\left(\dot f_{\theta,r_n}(t)-\dot f_{\theta}(t)\right)^T(f(t)- f_{\bm\theta}(t))g(t)dt\nonumber\\
&&+\int_0^1\left(\dot f_{\theta,r_n}(t)\right)^T(f_{\bm\theta}(t)-f_{\bm\theta,r_n}(t))g(t)dt=\bm 0\nonumber.
\end{eqnarray}
Subtracting \eqref{theta0} from the above equation we get
\begin{eqnarray}
\lefteqn{\int_0^1\left(\dot f_{\theta_0}(t)\right)^T(f(t)-f_0(t))g(t)dt
-\int_0^1\left(\dot f_{\theta_0}(t)\right)^T(f_{\bm\theta}(t)-f_{\bm\theta_0}(t))g(t)dt}\nonumber\\
&&+\int_0^1\left(\dot f_{\theta}(t)-\dot f_{\theta_0}(t)\right)^T(f(t)- f_{\bm\theta}(t))g(t)dt\nonumber\\
&&+\int_0^1\left(\dot f_{\theta,r_n}(t)-\dot f_{\theta}(t)\right)^T(f(t)- f_{\bm\theta}(t))g(t)dt\nonumber\\
&&+\int_0^1\left(\dot f_{\theta,r_n}(t)\right)^T(f_{\bm\theta}(t)-f_{\bm\theta,r_n}(t))g(t)dt=\bm 0\nonumber.
\end{eqnarray}
Replacing the difference between the values of a function at two different values of an argument by the integral of the corresponding partial derivative, we get
\begin{eqnarray}
\lefteqn{\bm M(f,\bm\theta)(\bm\theta-\bm\theta_0)}\nonumber\\
&=&\int_0^1\left(\dot f_{\theta_0}(t)\right)^T(f(t)- f_0(t))g(t)dt\nonumber\\
&&+\int_0^1\left(\dot f_{\theta,r_n}(t)-\dot f_{\theta}(t)\right)^T(f(t)-f_{\bm\theta}(t))g(t)dt\nonumber\\
&&+\int_0^1\left(\dot f_{\theta,r_n}(t)\right)^T(f_{\bm\theta}(t)-f_{\bm\theta,r_n}(t))g(t)dt,\nonumber
\end{eqnarray}
where $\bm M(f,\bm\theta)$ is given by
\begin{eqnarray}
\lefteqn{\int_0^1\int_0^1\left(\dot f_{\theta_0}(t)\right)^T\dot f_{\theta_0+\lambda(\bm\theta-\bm\theta_0)}(t)d\lambda\, g(t)dt}\nonumber\\
&&-\int_0^1\int_0^1\ddot f_{\bm\theta_0+\lambda(\bm\theta-\bm\theta_0)}(t)(f_0(t)-f_{\theta_0}(t))d\lambda\, g(t)dt\nonumber\\
&&-\int_0^1\int_0^1\ddot f_{\bm\theta_0+\lambda(\bm\theta-\bm\theta_0)}(t)(f_{\bm\theta_0}(t)-f_{\bm\theta}(t))d\lambda g(t)dt\nonumber\\
&&-\int_0^1\int_0^1\ddot f_{\bm\theta_0+\lambda(\bm\theta-\bm\theta_0)}(t)(f(t)-f_0(t))d\lambda\, g(t)dt\nonumber\\
&&-\int_0^1\int_0^1\left(\dot f_{\theta}(t)-\dot f_{\theta_0}(t)\right)^T\dot f_{\theta_0+\lambda(\bm\theta-\bm\theta_0)}(t)d\lambda\, g(t)dt.\nonumber
\end{eqnarray}
For a sequence $\epsilon\rightarrow0$, define
\begin{eqnarray}
E_n=\{(f,\bm\theta):\sup_{t\in[0,1]}|f(t)-f_{0}(t)|\leq\epsilon_n, \|\bm\theta-\bm\theta_0\|\leq\epsilon_n\}.\nonumber
\end{eqnarray}
By Lemmas $7$ and $8$, we can choose $\epsilon_n$ so that $\Pi^*_n(E^c_n|\bm X,\bm Y,\sigma^2)=o_{P_0}(1)$. Then, $\bm M(f,\bm\theta)$ is invertible and the eigenvalues of $[\bm M(f,\bm\theta)]^{-1}$ are bounded away from $0$ and $\infty$ for sufficiently large $n$ and $\|(\bm M(f,\bm\theta))^{-1}-\bm J^{-1}_{\bm\theta_0}\|=o(1)$ for $(f, \bm\theta)\in E_n$. Using \eqref{rk4}, on $E_n$
\begin{eqnarray}
\lefteqn{\sqrt{n}(\bm\theta-\bm\theta_0)}\nonumber\\
&=&\left(\bm J^{-1}_{\bm\theta_0}+o(1)\right)\left(\sqrt{n}\int_0^1\left(\dot f_{\theta_0}(t)\right)^T(f(t)-f_0(t))g(t)dt+O(\sqrt{n}r^{-4}_n)\right).\nonumber
\end{eqnarray}
Note that $\sqrt{n}\bm J_{\bm\theta_0}\left(\bm\Gamma(f)-\bm\Gamma(f_0)\right)=\sqrt{n}\bm H_{n}^T\bm\beta-\sqrt{n}\bm J^{-1}_{\bm\theta_0}\bm\Gamma(f_0)$. It was shown in the proof of Theorem $4.2$ that for a given $\sigma^2$, the total variation distance between the posterior distribution of $\sqrt{n}\bm H_{n}^T\bm\beta-\sqrt{n}\bm J^{-1}_{\bm\theta_0}\bm\Gamma(f_0)$ and $N(\bm\mu_n,\sigma^2\bm\Sigma_n)$ converges in $P_0$-probability to $0$. By Lemma $10$, both $\bm\mu_n$ and $\bm\Sigma_n$ are stochastically bounded. Thus the posterior distribution of $\bm J^{-1}_{\bm\theta_0}\sqrt{n}\left(\bm\Gamma(f)-\bm\Gamma(f_0)\right)$ assigns most of its mass inside a large compact set with high true probability.
\end{proof}
The next lemma describes the asymptotic behavior of the mean and variance of the limiting normal distribution given by Theorem 4.2.

%%%%%%%%%%Proof of Lemma 10

\begin{lemma}
The mean and variance of the limiting normal approximation given by Theorem 4.2 are stochastically bounded.
\end{lemma}
\begin{proof}
First we study the asymptotic behavior of the matrix $\Var(\bm\mu_n|\bm X)=\bm\Sigma_n=n\bm H^T_n(\bm X^T_n\bm X_n)^{-1}\bm H_n$.
If $C_{k}(\cdot)\in C^{m^*}(0,1)$ for some $1\leq m^*<m$, then by equation $(2)$ of \citet[page 167]{de1978practical}, we have for all $k=1,\ldots,p$, $$\sup\{|C_{k}(t)-\tilde C_{k}(t)|:t\in[0,1]\}=O(k_n^{-1}),$$ where $\tilde C_{k}(\cdot)=\bm \alpha_{k}^T\bm N(\cdot)$ and $\bm\alpha^T_{k}=(C_{k}(t^*_1),\ldots,C_{k}(t^*_{k_n+m-1}))$ with appropriately chosen $t^*_1,\ldots,t^*_{k_n+m-1}$.
We can write $\bm H^T_{n}\left(\bm X^T_n\bm X_n\right)^{-1}\bm H_{n}$ as
\begin{eqnarray}
\lefteqn{(\bm H_{n}-\tilde{\bm H}_{n})^T\left(\bm X^T_n\bm X_n\right)^{-1}(\bm H_{n}-\tilde{\bm H}_{n})
+\tilde{\bm H}^T_{n}\left(\bm X^T_n\bm X_n\right)^{-1}\tilde{\bm H}_{n}}\nonumber\\
&&+(\bm H_{n}-\tilde{\bm H}_{n})^T\left(\bm X^T_n\bm X_n\right)^{-1}\tilde{\bm H}_{n}+\tilde{\bm H}^T_{n}\left(\bm X^T_n\bm X_n\right)^{-1}(\bm H_{n}-\tilde{\bm H}_{n}),\nonumber
\end{eqnarray}
where the $k^{th}$ row of $\tilde{\bm H}^T_{n}$ is given by $\int_0^1\tilde{C_{k}}(t)(\bm N(t))^Tg(t)dt$ for $k=1,\ldots,p.$ Let us denote $\bm A=(\bm\alpha_1,\ldots,\bm\alpha_p)$. Then
\begin{eqnarray}
\lefteqn{\tilde{\bm H}^T_{n}\left(\bm X^T_n\bm X_n\right)^{-1}\tilde{\bm H}_{n}}\nonumber\\
&=&n^{-1}\bm A^T\left(\int_0^1\bm N(t)\bm N^T(t)g(t)dt\right)\left(\frac{\bm X^T_n\bm X_n}{n}\right)^{-1}\left(\int_0^1\bm N(t)\bm N^T(t)g(t)dt\right)\bm A.\nonumber
\end{eqnarray}
We show that
\begin{eqnarray}
\lefteqn{\bm A^T\left(\int_0^1\bm N(t)\bm N^T(t)g(t)dt\right)\left(\frac{\bm X^T_n\bm X_n}{n}\right)^{-1}\left(\int_0^1\bm N(t)\bm N^T(t)g(t)dt\right)\bm A}\nonumber\\
&&-\bm A^T\left(\int_0^1\bm N(t)\bm N^T(t)g(t)dt\right)\bm A\nonumber
\end{eqnarray}
converges in $P_0$-probability to the null matrix of order $p$. For a $\bm l\in\mathbb{R}^p$, let $\bm c=\left(\int_0^1\bm N(t)\bm N^T(t)g(t)dt\right)\bm A\bm l$.
Then we can write $$\bm l^T\bm A^T\left(\int_0^1\bm N(t)\bm N^T(t)g(t)dt\right)\left(\frac{\bm X^T_n\bm X_n}{n}\right)^{-1}\left(\int_0^1\bm N(t)\bm N^T(t)g(t)dt\right)\bm A\bm l$$ as $\bm c^T\left(\frac{\bm X^T_n\bm X_n}{n}\right)^{-1}\bm c$. Let us denote by $Q_n$ the empirical distribution function of $X_1,\ldots,X_n$. Note that
\begin{eqnarray}
\lefteqn{\left|\bm c^T\left(\frac{\bm X^T_n\bm X_n}{n}\right)\bm c-\bm c^T\left(\int_0^1\bm N(t)\bm N^T(t)g(t)dt\right)\bm c\right|}\nonumber\\
&\leq&\sup_{t\in[0,1]}|Q_n(t)-G(t)|\bm c^T\bm c\sup_{t\in[0,1]}\|\bm N(t)\|^2\nonumber\\
&=&O_{P_0}(n^{-1/2})\bm c^T\bm c\nonumber\\
&=&O_{P_0}(n^{-1/2}k_n)\bm c^T\left(\int_0^1\bm N(t)\bm N^T(t)g(t)dt\right)\bm c\nonumber,
\end{eqnarray}
%\end{eqnarray}
the third step following from Donsker's Theorem and the fact that $$\sup_{t\in[0,1]}\|\bm N(t)\|^2\|\leq1.$$ In the last step we used the fact that the eigenvalues of the matrix $\int_0^1\bm N(t)\bm N^T(t)g(t)dt$ are $O(k^{-1}_n)$ as proved in Lemma 6.1 of \citet{zhou1998local}. In that same lemma it was also proved that the eigenvalues of the matrix $\left(\bm X^T_n\bm X_n/n\right)$ are $O_{P_0}(k^{-1}_n)$. Both these results are applied in the fourth step of the next calculation. Using the fact that $\|\bm R^{-1}-\bm S^{-1}\|\leq\|\bm S^{-1}\|\|\bm R-\bm S\|\|\bm S^{-1}\|$ for two nonsingular matrices $\bm R$ and $\bm S$ of the same order, we get
\begin{eqnarray}
\lefteqn{\left|\bm c^T\left(\frac{\bm X^T_n\bm X_n}{n}\right)^{-1}\bm c-\bm c^T\left(\int_0^1\bm N(t)\bm N^T(t)g(t)dt\right)^{-1}\bm c\right|}\nonumber\\
&=&O_{P_0}(n^{-1/2}k_n)\bm c^T\left(\frac{\bm X^T_n\bm X_n}{n}\right)^{-1}\bm c\nonumber\\
&=&O_{P_0}(n^{-1/2}k_n)\bm l^T\bm A^T\left(\int_0^1\bm N(t)\bm N^T(t)g(t)dt\right)\left(\frac{\bm X^T_n\bm X_n}{n}\right)^{-1}\nonumber\\
&&\times\left(\int_0^1\bm N(t)\bm N^T(t)g(t)dt\right)\bm A\bm l\nonumber\\
&=&O_{P_0}(n^{-1/2}k_n)k^{-1}_n\bm l^T\bm A^T\bm A\bm l=o_{P_0}(1).\nonumber
\end{eqnarray}
Now note that the $(i,j)^{th}$ element of the $p\times p$ matrix $\bm A^T\left(\int_0^1\bm N(t)\bm N^T(t)g(t)dt\right)\bm A$ is given by $\int_0^1\tilde{C}_i(t)\tilde{C}_j(t)g(t)dt$, which converges to $\int_0^1{C}_i(t){C}_j(t)g(t)dt$, the $(i,j)^{th}$ element of the matrix $\int_0^1{\bm C}(t){\bm C}^T(t)g(t)dt$ which is $\bm J^{-1}_{\bm\theta_0}\int_0^1\left(\dot{f}_{\bm\theta_0}(t)\right)^T\dot{f}_{\bm\theta_0}(t)g(t)dt\left(\bm J^{-1}_{\bm\theta_0}\right)^T$.
Let us denote by $\bm 1_{k_n+m-1}$ the $(k_n+m-1)$-component vector with all elements $1$. Then for $k=1,\ldots,p$, the $k^{th}$ diagonal entry of the matrix $(\bm H_{n}-\tilde{\bm H}_{n})^T(\bm X^T_n\bm X_n)^{-1}(\bm H_{n}-\tilde{\bm H}_{n})$ is given by
\begin{eqnarray}
\lefteqn{\int_0^1(C_{k}(t)-\tilde C_{k}(t))(\bm N(t))^Tg(t)dt\,{({\bm X^T_n\bm X_n})}^{-1}\int_0^1(C_{k}(t)-\tilde C_{k}(t))(\bm N(t))g(t)dt}\nonumber\\
&=&\frac{1}{n}\int_0^1(C_{k}(t)-\tilde C_{k}(t))(\bm N(t))^Tg(t)dt\,{\left({\bm X^T_n\bm X_n}/{n}\right)}^{-1}\nonumber\\
&&\times\int_0^1(C_{k}(t)-\tilde C_{k}(t))\bm N(t)g(t)dt\nonumber\\
&\asymp&\frac{k_n}{n}\int_0^1(C_{k}(t)-\tilde C_{k}(t))(\bm N(t))^Tg(t)dt\int_0^1(C_{k}(t)-\tilde C_{k}(t))\bm N(t)g(t)dt\nonumber\\
&\lesssim&\frac{1}{nk_n},\nonumber
\end{eqnarray}
the last step following by the application of the Cauchy-Schwarz inequality and the facts that $\sup\{|C_{k}(t)-\tilde C_{k}(t)|:t\in[0,1]\}=O(k_n^{-1})$ and $\int_0^1\|\bm N(t)\|^2dt\leq 1$. Thus, the eigenvalues of $(\bm H_{n}-\tilde{\bm H}_{n})^T(\bm X^T_n\bm X_n)^{-1}(\bm H_{n}-\tilde{\bm H}_{n})$ are of the order $(nk_n)^{-1}$ or less.
Hence,
\begin{eqnarray}
n\bm H^T_{n}(\bm X^T_n\bm X_n)^{-1}\bm H_{n}&\stackrel{P_0}\rightarrow&\bm J^{-1}_{\bm\theta_0}\int_0^1\left(\dot{f}_{\bm\theta_0}(t)\right)^T\dot{f}_{\bm\theta_0}(t)g(t)dt\,\left(\bm J^{-1}_{\bm\theta_0}\right)^T.\nonumber
\end{eqnarray}
Thus, the eigenvalues of $\bm\Sigma_n$ are stochastically bounded. Now
note that
\begin{eqnarray}
\E(\bm\mu_n|\bm X)&=&\sqrt{n}\bm H^T_n(\bm X^T_n\bm X_n)^{-1}\bm X^T_nf_0(\bm X)-\sqrt{n}\bm J^{-1}_{\bm\theta_0}\bm\Gamma(f_0)\nonumber\\
&=&\sqrt{n}\bm H^T_n(\bm X^T_n\bm X_n)^{-1}\bm X^T_n(f_0(\bm X)-\bm X_n\bm\beta^*)\nonumber\\
&&+\sqrt{n}\int_0^1\bm C(t)(\bm N^T(t)\bm\beta^*-f_0(t))g(t)dt.\nonumber
\end{eqnarray}
Using the Cauchy-Schwarz inequality and \eqref{spldis}, we get
\begin{eqnarray}
\|\E(\bm\mu_n|\bm X)\|&\lesssim&\sqrt{n}\,\mathrm{maxeig}(\bm H^T_n(\bm X^T_n\bm X_n)^{-1}\bm H_n)^{1/2}\sqrt{n}k^{-m}_n+\sqrt{n}k^{-m}_n\nonumber\\
&=&O_{P_0}(\sqrt{n}k^{-m}_n)=o_{P_0}(1).\nonumber
\end{eqnarray}
Thus, $Z_n:=\|\E(\bm\mu_n|\bm X)\|^2+\mathrm{maxeig}(\Var(\bm\mu_n|\bm X))$ is stochastically bounded. Given $M>0$, there exists $L>0$ such that $\sup_nP_0(Z_n>L)<M^{-2}$. Hence for all $n$, $P_0\left(\|\bm\mu_n\|>M\right)$ is bounded above by $M^{-2}\E_0\left[\E\left(\|\bm\mu_n\|^2|\bm X\right)\mathrm{l}\!\!\!1\{Z_n\leq L\}\right]+P_0(Z_n>L)$ which is less than or equal to $(L+1)/M^2$. Hence, $\bm\mu_n$ is stochastically bounded.
\end{proof}
In the next lemma we establish the posterior consistency of $\sigma^2$.

%%%%%%%%%%Proof of Lemma 11

\begin{lemma}
For all $\epsilon>0$, we have $\Pi^*_n(|\sigma^2-\sigma^2_0|>\epsilon|\bm X,\bm Y)=o_{P_0}(1)$.
\end{lemma}
\begin{proof}
The joint density of $\bm Y$, $\bm\beta$ and $\sigma^2$ is proportional to
\begin{eqnarray}
\lefteqn{\sigma^{-n}\exp\left\{-\frac{1}{2\sigma^2}(\bm Y-\bm X_n\bm\beta)^T(\bm Y-\bm X_n\bm\beta)\right\}}\nonumber\\
&&\times\sigma^{-k_n-m+1}\exp\left\{-\frac{1}{2n^2k^{-1}_n\sigma^2}\bm\beta^T\bm\beta\right\}\exp\left(-\frac{b}{\sigma^2}\right)(\sigma^2)^{-a-1},\nonumber
\end{eqnarray}
which implies that the posterior distribution of $\sigma^2$ is inverse gamma with shape parameter $n/2+a$ and scale parameter $$\frac{1}{2}\left\{\bm Y^T\bm Y-\bm Y^T\bm X_n(\bm X^T_n\bm X_n+k_nn^{-2}\bm I_{k_n+m-1})^{-1}\bm X^T_n\bm Y\right\}+b.$$ Hence, the posterior mean of $\sigma^2$ is given by
\begin{eqnarray}
\E(\sigma^2|\bm X,\bm Y)=\frac{\frac{1}{2}\left\{\bm Y^T\bm Y-\bm Y^T\bm X_n(\bm X^T_n\bm X_n+k_nn^{-2}\bm I_{k_n+m-1})^{-1}\bm X^T_n\bm Y\right\}+b}{n/2+a-1},\nonumber
\end{eqnarray}
which behaves like the $n^{-1}\left({\bm Y^T\bm Y-\bm Y^T\bm X_n(\bm X^T_n\bm X_n+k_nn^{-2}\bm I_{k_n+m-1})^{-1}\bm X^T_n\bm Y}\right)$ asymptotically. The later can be written as
\begin{eqnarray}
n^{-1}\left({\bm Y^T(\bm I_{n}-\bm P_{\bm X_n})\bm Y+\bm Y^T(\bm P_{\bm X_n}-\bm X_n(\bm X^T_n\bm X_n+k_nn^{-2}\bm I_{k_n+m-1})^{-1}\bm X^T_n)\bm Y}\right),\nonumber
\end{eqnarray}
where $\bm P_{\bm X_n}=\bm X_n(\bm X^T_n\bm X_n)^{-1}\bm X^T_n$. We will show that $n^{-1}\bm Y^T(\bm I_n-\bm P_{\bm X_n})\bm Y\stackrel{P_0}\rightarrow\sigma^2_0$ and $n^{-1}\bm Y^T(\bm P_{\bm X_n}-\bm X_n(\bm X^T_n\bm X_n+k_nn^{-2}\bm I_{k_n+m-1})^{-1}\bm X^T_n)\bm Y=o_{P_0}(1)$ and hence $\E(\sigma^2|\bm X,\bm Y)\stackrel{P_0}\rightarrow\sigma^2_0$. Using $\bm Y=f_0(\bm X)+\bm\varepsilon$, we note that
\begin{eqnarray}
\bm Y^T(\bm I_n-\bm P_{\bm X_n})\bm Y&=&\bm\varepsilon^T(\bm I_n-\bm P_{\bm X_n})\bm\varepsilon+f_0(\bm X)^T(\bm I_n-\bm P_{\bm X_n})f_0(\bm X)\nonumber\\
&&+2\bm\varepsilon^T(\bm I_n-\bm P_{\bm X_n})f_0(\bm X).\nonumber
\end{eqnarray}
We show that $\bm\varepsilon^T(\bm I_n-\bm P_{\bm X_n})\bm\varepsilon/n\stackrel{P_0}\rightarrow\sigma^2_0,\,n^{-1}f_0(\bm X)^T(\bm I_n-\bm P_{\bm X_n})f_0(\bm X)=o_{P_0}(1)$ and $n^{-1}\bm\varepsilon^T(\bm I_n-\bm P_{\bm X_n})f_0(\bm X)=o_{P_0}(1)$. Now, $\E_0\left(\bm\varepsilon^T(\bm I_n-\bm P_{\bm X_n})\bm\varepsilon/n\right)\rightarrow\sigma_0^2$ as $n\rightarrow\infty$. Also,
\begin{eqnarray}
\Var_0\left(\bm\varepsilon^T(\bm I_n-\bm P_{\bm X_n})\bm\varepsilon/n\right)&=&n^{-2}\left(\E_0\Var\left(\bm\varepsilon^T(\bm I_n-\bm P_{\bm X_n})\bm\varepsilon|\bm X\right)\right.\nonumber\\
&&\left.+\Var_0\E(\bm\varepsilon^T(\bm I_n-\bm P_{\bm X_n})\bm\varepsilon|\bm X)\right).\nonumber
\end{eqnarray}
Now
\begin{eqnarray}
\Var(\bm\varepsilon^T(\bm I_n-\bm P_{\bm X_n})\bm\varepsilon|\bm X)&=&(\mu_4-\sigma^2_0)(n-k_n-m+1)\nonumber\\
\E(\bm\varepsilon^T(\bm I_n-\bm P_{\bm X_n})\bm\varepsilon|\bm X)&=&\sigma^2_0(n-k_n-m+1),\nonumber
\end{eqnarray}
$\mu_4$ being the fourth order central moment of $\varepsilon_i$ for $i=1,\ldots,n$. Hence, $$\Var_0\left(\bm\varepsilon^T(\bm I_n-\bm P_{\bm X_n})\bm\varepsilon/n\right)\rightarrow0\, \text{as}\, n\rightarrow\infty.$$ Thus, $\bm\varepsilon^T(\bm I_n-\bm P_{\bm X_n})\bm\varepsilon/n\stackrel{P_0}\rightarrow\sigma^2_0$. We can write for $\bm\beta^*$ satisfying \eqref{spldis}
\begin{eqnarray}
f_0(\bm X)^T(\bm I_n-\bm P_{\bm X_n})f_0(\bm X)&=&(f_0(\bm X)-\bm X_n\bm\beta^*)^T(\bm I_n-\bm P_{\bm X_n})(f_0(\bm X)-\bm X_n\bm\beta^*)\nonumber\\
&\lesssim&nk^{-2m}_n,\nonumber
\end{eqnarray}
since $(\bm I_n-\bm P_{\bm X_n})\bm X_n=\bm 0$. Using the Cauchy-Schwarz inequality, we get
\begin{eqnarray}
\left|n^{-1}\bm\varepsilon^T(\bm I_n-\bm P_{\bm X_n})f_0(\bm X)\right|&=&\left|n^{-1}\bm\varepsilon^T(\bm I_n-\bm P_{\bm X_n})(f_0(\bm X)-\bm X_n\bm\beta^*)\right|\nonumber\\
&\leq&\sqrt{\bm\varepsilon^T\bm\varepsilon/n}k^{-m}_n=o_{P_0}(1).\nonumber
\end{eqnarray}
By the Binomial Inverse Theorem,
\begin{eqnarray}
\bm P_{\bm X_n}-\bm X_n(\bm X^T_n\bm X_n+k_nn^{-2}\bm I_{k_n+m-1})^{-1}\bm X^T_n&=&k_nn^{-2}\bm X_n(\bm X^T_n\bm X_n)^{-1}\nonumber\\
&&\times\left(\bm I_{k_n+m-1}+(\bm X^T_n\bm X_n)^{-1}\frac{k_n}{n^2}\right)^{-1}\nonumber\\
&&\times(\bm X^T_n\bm X_n)^{-1}\bm X^T_n\nonumber
\end{eqnarray}
whose eigenvalues are of the order
$k_nn^{-2}{n}{k^{-1}_n}{k^2_n}{n^{-2}}={k^2_n}{n^{-3}}$.
Hence, the random variable $\bm Y^T\left(\bm P_{\bm X_n}-\bm X_n(\bm X^T_n\bm X_n+k_nn^{-2}\bm I_{k_n+m-1})^{-1}\bm X^T_n\right)\bm Y/n$ converges in $P_0$-probability to $0$ and $\E(\sigma^2|\bm X,\bm Y)\stackrel{P_0}\rightarrow\sigma^2_0$. Also, $$\Var(\sigma^2|\bm X, \bm Y)=\left(\E(\sigma^2|\bm X,\bm Y)\right)^2/(n/2+a-2)=o_{P_0}(1).$$ By using the Markov's inequality, we finally get $\Pi^*_n(|\sigma^2-\sigma^2_0|>\epsilon|\bm X,\bm Y)=o_{P_0}(1)$ for all $\epsilon>0$.
\end{proof}

\section{Appendix}

The following result was used to prove Lemmas 6, 7 and 10.
\begin{lemma}
For any $0\leq r\leq m-2$, there exist constants $L_{\max}>L_{\min}>0$ such that uniformly in $t\in[0,1]$, $$\frac{L_{\min}\sigma^2k^{2r+1}_n}{n}\left(1+o_{P_0}(1)\right)\leq\left(\bm N^{(r)}(t)\right)^T\left(\bm X^T_n\bm X_n\right)^{-1}\bm N^{(r)}(t)\leq\frac{L_{\max}\sigma^2k^{2r+1}_n}{n}\left(1+o_{P_0}(1)\right).$$
\end{lemma}
The proof is implicit in Lemma 5.4 of \citet{zhou2000derivative}.
\bibliographystyle{chicago}
\bibliography{ref}

\end{document}